\documentclass[noams]{compositio}
\usepackage{mathrsfs}
\usepackage{bbm}%
\usepackage{bbding}
\usepackage[mathcal]{euscript}
\usepackage{graphicx}
\usepackage{amscd}
\usepackage{calc}
\usepackage{mathrsfs,dsfont}
\usepackage{CJK,fancyhdr}
\usepackage{amsmath,amssymb,amsfonts}
\usepackage{epsfig,enumerate}
\usepackage{latexsym}
\usepackage{indentfirst, latexsym}
\usepackage{graphics}
\usepackage[all,poly,knot]{xy}
\usepackage[pagebackref]{hyperref}

\newtheorem{theorem} {Theorem} [section]
\newtheorem{proposition}[theorem]{Proposition}
\newtheorem{corollary}  [theorem]     {Corollary}
\newtheorem{lemma}  [theorem]     {Lemma}
\newtheorem{example}  [theorem]     {Example}
\newtheorem{question}  [theorem]     {Question}

\newtheorem{remark}  [theorem]     {Remark}
\newtheorem{definition}  [theorem]     {Definition}

\newtheorem{observation}  [theorem]     {Observation}

\renewcommand{\contentsname}{Contents}

\newcommand{\w}{\wedge}

\renewcommand{\1}{\mathds{1}}
\newcommand{\G}{\mathbb{G}}
\newcommand{\db}{\overline{\partial}}

\newcommand{\lc}{\lrcorner}

\newcommand{\Om}{\Omega}
\newcommand{\btheorem}{\begin{theorem}}
\newcommand{\etheorem}{\end{theorem}}
\newcommand{\bproposition}{\begin{proposition}}
\newcommand{\eproposition}{\end{proposition}}
\newcommand{\bdefinition}{\begin{definition}}
\newcommand{\edefinition}{\end{definition}}
\newcommand{\bcorollary}{\begin{corollary}}
\newcommand{\ecorollary}{\end{corollary}}
\newcommand{\bproof}{\begin{proof}}
\newcommand{\eproof}{\end{proof}}
\newcommand{\beq}{\begin{equation}}
\newcommand{\eeq}{\end{equation}}
\newcommand{\ee}{\end{eqnarray*}}
\newcommand{\be}{\begin{eqnarray*}}

\newcommand{\elemma}{\end{lemma}}
\newcommand{\blemma}{\begin{lemma}}

\renewcommand{\>}{\rightarrow}

\newcommand{\p}{\partial}

\newcommand{\im}{\textmd{im}}

\def\p{\partial}
\def\o{\overline}
\def\b{\bar}

\def\mc{\mathcal}

\def\w{\wedge}

\def\l{\lrcorner}

\begin{document}

\title{On local stabilities of $p$-K\"{a}hler structures}
\author{Sheng Rao}
\email{likeanyone@whu.edu.cn, raoshengmath@gmail.com}
\address{School of Mathematics and Statistics, Wuhan University,
Wuhan 430072, China.}
\author{Xueyuan Wan}
\email{xwan@chalmers.se}
\address{Mathematical Sciences, Chalmers University of Technology, University of Gothenburg, 412 96 Gothenburg, Sweden.}
\author{Quanting Zhao}
\email{zhaoquanting@126.com; zhaoquanting@mail.ccnu.edu.cn}
\address{School of Mathematics and Statistics \&
Hubei Key Laboratory of Mathematical Sciences, Central China Normal
University, Wuhan 430079, China.}

\thanks{Rao is partially supported by NSFC (Grant No. 11671305, 11771339).
Zhao is partially supported by China Postdoctoral Science
Foundation and NSFC (Grant No. 2016M592356 and 11801205).}

\date{\today}

\begin{abstract}
By use of a natural extension map and a power series method,
we obtain a local stability theorem for $p$-K\"{a}hler structures with the
$(p,p+1)$-th mild $\p\db$-lemma under small differentiable deformations.
\end{abstract}

\maketitle

\vspace*{6pt}\tableofcontents  

\section{Introduction}\label{intro}
Local stabilities of complex structures are important topics in deformation theory of complex structures. We will prove local stabilities of $p$-K\"{a}hler structures with the $(p,p+1)$-th mild $\p\b{\p}$-lemma by power series method, initiated by Kodaira-Nirenberg-Spencer \cite{kns} and Kuranishi \cite{ku}.
\begin{theorem}\label{mt}
For any positive integer $p\leq n-1$, any
small differentiable deformation $X_t$ of  an $n$-dimensional
$p$-K\"ahler manifold $X_0$ satisfying  the
$(p,p+1)$-th mild $\p\db$-lemma is still
$p$-K\"ahlerian.
\end{theorem}
Here the
\emph{$(p,p+1)$-th mild $\p\db$-lemma} for a complex manifold means that each $\db$-closed $\p$-exact $(p,p+1)$-form on this manifold is $\p\db$-exact, which is a new notion generalizing the
$(n-1,n)$-th one first introduced in \cite{RwZ}. A complex manifold is \emph{$p$-K\"{a}hlerian} if it admits a \emph{$p$-K\"{a}hler form}, i.e., a $d$-closed transverse $(p,p)$-form as in Definition \ref{pkf}.

Recall the fact that each $n$-dimensional complex manifold is $n$-K\"ahlerian and two basic properties of $p$-K\"ahlerian structures:
\begin{lemma}[{\cite[Proposition 1.15]{aa} and also \cite[Corollary 4.6]{RwZ}}]\label{pkb}
A complex manifold $M$ is $1$-K\"ahler if and only if $M$ is
K\"ahler;  an $n$-dimensional complex manifold $M$ is $(n-1)$-K\"ahler if and only if $M$ is \emph{balanced}, i.e., it admits a real positive $(1,1)$-form $\omega$, satisfying
$$d(\omega^{n-1})=0.$$
\end{lemma}

Thus, as a direct corollary of Theorem \ref{mt}, we obtain:
\begin{corollary}\label{cor-mt}
Let $\pi: \mathcal{X} \rightarrow B$ be a differentiable family of compact complex manifolds.
\begin{enumerate}[$(i)$]
    \item {{\emph{\textbf{$($\cite[Theorem 15]{KS}$)$}}}\label{stab-Kahler}
If a fiber $X_0:= \pi^{-1}(t_0)$ admits a K\"ahler metric, then, for a sufficiently small neighborhood $U$ of $t_0$ on $B$, the fiber $X_t:=\pi^{-1}(t)$ over any point $t\in U$ still admits a K\"ahler metric, which depends smoothly on $t$ and coincides for $t=t_0$ with the given K\"ahler metric on $X_0$.}
    \item {\emph{\textbf{$($\cite[Theorem 1.5]{RwZ}$)$}} \label{balanced-Kahler}
    Let $X_0$ be a balanced manifold of complex dimension $n$,
satisfying the $(n-1,n)$-th mild $\p\db$-lemma. Then $X_t$ also
admits a balanced metric for $t$ small.}
 \end{enumerate}
\end{corollary}

The first assertion of Corollary \ref{cor-mt} is the fundamental Kodaira-Spencer's local stability theorem of K\"ahler structure, and motivates the second one of Corollary \ref{cor-mt} and many other related works on local stabilities of
complex structures in \cite{FY,V,w,au,AU}. The counter-example of L. Alessandrini-G. Bassanelli \cite{ab} tells us that the result in the second assertion of Corollary \ref{cor-mt} does not necessarily hold without the $(n-1,n)$-th mild $\p\db$-lemma assumption.

In Section \ref{dpm}, we will study the difference between the
$(p,q)$-th mild $\p\db$-lemma and other versions of $\p\db$-lemmata in the roles of Theorem \ref{mt}, and the modification stability of the
$(p,q)$-th mild $\p\db$-lemma by Proposition \ref{mefm}, which provides us more classes of complex manifolds to admit the $(p,q)$-th mild $\p\db$-lemma.
Here the (standard) \emph{$\p\b{\p}$-lemma} refers
to: for every pure-type $d$-closed form on a complex manifold, the properties of
$d$-exactness, $\p$-exactness, $\b{\p}$-exactness and
$\p\b{\p}$-exactness are equivalent, while its variants are described by Subsection \ref{vddbar}.
Obviously, one has the implication hierarchy on a complex $n$-dimensional manifold for any positive integer $p\leq n-1$:
\begin{align}
&\text{the $\p\db$-lemma} \nonumber\\
\Longrightarrow\ &\text{the $(p, p + 1)$-th strong $\p\db$-lemma}\label{s2s}\\
\Longrightarrow\ &\text{the $(p, p + 1)$-th mild  $\p\db$-lemma}\label{s2m}\\
\Longrightarrow\ &\text{the $(p, p + 1)$-th weak  $\p\db$-lemma}\label{m2w}.
\end{align}
For $p=n-1$, the implication hierarchy is strict: \cite[Example 4.10]{au} is the one, satisfying the strong $\p\db$-lemma but not the standard one; the nilmanifold endowed with a left-invariant abelian complex structure of dimension $2n$ or a left-invariant non-nilpotent balanced complex structure of complex dimension $3$ by \cite[Proposition 3.8 and Corollary 3.4]{RwZ} and \cite[Proposition 2.9]{AU} distinguishes the mild and strong $\p\db$-lemmata; and the weak $\p\db$-lemma holds on the complex three-dimensional Iwasawa manifold \cite[Example 3.7]{RwZ} but the mild one fails. Moreover, we construct a new ten-dimensional balanced nilmanifold in Example \ref{bcvary} for the strictness of Implication \ref{s2m}, which satisfies the $(4, 5)$-th mild but not strong $\p\b{\p}$-lemma and also the deformation variance of the $(4, 4)$-th Bott-Chern numbers.
Motivated by these, one is natural to ask:
\begin{question}
\emph{Find an $n$-dimensional complex manifold or in particular a $p$-K\"ahler manifold
such that one of Implications \eqref{s2s}, \eqref{s2m}, \eqref{m2w} is strict for each positive integer $p< n-1$.}
\end{question}

Now let us describe our approach to prove
local stability of $p$-K\"{a}hler structures. An application of Kuranishi's completeness theorem \cite{ku} reduces our power series proof to the Kuranishi family $\varpi:\mc{K}\to T$, that is, we will construct a natural $p$-K\"{a}hler extension
$\tilde{\omega}_t$ of the $p$-K\"{a}hler form $\omega_0$ on $X_0$, such that $\tilde{\omega}_t$ is a
$p$-K\"{a}hler form on the general fiber $\varpi^{-1}(t)=X_t$. More
precisely, the extension is given by
$$e^{\iota_{\varphi}|\iota_{\o{\varphi}}}: A^{p,p}(X_0)\to A^{p,p}(X_t),\quad \omega_0\to \tilde{\omega}_t:=e^{\iota_{\varphi}|\iota_{\o{\varphi}}}(\omega(t)),$$
where $\omega(t)$ is a family of smooth $(p,p)$-forms to be constructed on $X_0$,
depending smoothly on $t$, and $\omega(0)=\omega_0$. Here $\varphi$ is the family of Beltrami differentials induced by the Kuranishi family.
The extension map $e^{\iota_{\varphi}|\iota_{\o{\varphi}}}$ is first introduced in \cite{RZ2,RZ15} and given in Definition \ref{map}.

This method
is developed in
\cite{LSY,Sun,SY,lry,RZ,RZ2,RZ15,RwZ,lrw}. However, we have to solve many more equations in the system \eqref{1.4} here than in the balanced case \cite{RwZ}, which are much more difficult in essence. Fortunately, one is able to reduce this complicated system to that with only two ones as \eqref{1.6} by comparing the types of the forms in the system and the orders in the induction simultaneously. This crucial consideration is also important in the solution of this system.

In this approach, we will use the observation crucially:
\begin{proposition}[{\cite[Proposition 4.12]{RwZ}}]\label{ext-trans}
Let $\pi: \mathcal{X} \> B$ be a  differentiable family of compact complex $n$-dimensional manifolds and
$\Omega_t$  a family of real $(p,p)$-forms with $p< n$, depending smoothly on $t$. Assume that
$\Omega_0$ is a transverse $(p,p)$-form on $X_0$. Then $\Omega_t$ is also transverse on $X_t$ for small $t$.
\end{proposition}
This proposition actually shows that any smooth real extension of a transverse $(p,p)$-form is still transverse.
So the obstruction to extend a $d$-closed transverse $(p,p)$-form on a compact complex manifold
lies in the $d$-closedness, to be resolved in Theorem \ref{0pq-sm-ext} in a more general setting. The detailed proof of Main Theorem \ref{mt} is given in Section \ref{psp}.

  \begin{theorem}[=Theorem \ref{pq-sm-ext}]\label{0pq-sm-ext}
  If $X_0$ satisfies the $(p,q+1)$- and $(q,p+1)$-th mild $\p\b{\p}$-lemmata, then there is a $d$-closed $(p,q)$-form $\Omega(t)$ on $X_t$ depending smoothly on $t$ with $\Omega(0)=\Omega_0$ for any  $d$-closed $\Omega_0\in A^{p,q}(X_0)$.
  \end{theorem}
  \begin{remark}\label{n-1-d-clsd-ext}
  \emph{The case $p=q=n-1$ of Theorem \ref{0pq-sm-ext} implies that the dimension of the space of d-closed left-invariant $(n-1,n-1)$-forms
  on a $2n$-dimensional nilmanifold endowed with a left-invariant abelian complex structure is deformation invariant,
  where the $(n-1,n)$-th mild $\p\b{\p}$-lemma holds from \cite[Corollary 3.4]{RwZ}.}
  \end{remark}

In Section \ref{bcinv}, inspired by \cite{RZ15}, we will use Theorem \ref{0pq-sm-ext} to prove a result on deformation invariance of Bott-Chern numbers in Theorem \ref{inv-pq}.

This paper will follow the notations in \cite{lry,RZ15,RwZ}. All manifolds in this paper are assumed to be compact
complex $n$-dimensional manifolds. The symbol $A^{p,q}(X,E)$ stands for the space of
the holomorphic vector bundle $E$-valued $(p,q)$-forms on a
complex manifold $X$. We will always consider the differentiable family $\pi: \mathcal{X} \rightarrow
B$ of compact complex $n$-dimensional manifolds
over a sufficiently small domain in $\mathbb{R}^k$ with the
reference fiber $X_0:= \pi^{-1}(0)$ for the reference point $0$ and the general fibers $X_t:=
\pi^{-1}(t).$

\section{Deformation and $p$-K\"{a}hler structure}\label{dpm}
This section is to state some basics of analytic deformation theory of complex structures and the notion of $p$-K\"{a}hler structure.
\subsection{Deformation theory}
For holomorphic family of compact complex manifolds, we adopt the definition \cite[Definition 2.8]{k}; while for differentiable one, we adopt:
\begin{definition}[{\cite[Definition 4.1]{k}}] \emph{Let $\mc{X}$ be a differentiable manifold, $B$ a domain of $\mathbb{R}^k$ and $\pi$ a smooth map of $\mc{X}$ onto $B$.
By a \emph{differentiable family of $n$-dimensional compact complex manifolds} we mean the triple $\pi:\mc{X}\to B$ satisfying the following conditions:
\begin{enumerate}[$(i)$]
    \item
The rank of the Jacobian matrix of $\pi$ is equal to $k$ at every point of $\mc{X}$;
    \item
For each point $t\in B$, $\pi^{-1}(t)$ is a compact connected subset of $\mc{X}$;
    \item
$\pi^{-1}(t)$ is the underlying differentiable manifold of the $n$-dimensional compact complex manifold $X_t$ associated to each $t\in B$;
     \item
There is a locally finite open covering $\{\mathcal{U}_j\ |\ j=1,2,\cdots\}$ of $\mc{X}$ and complex-valued smooth functions $\zeta_j^1(p),\cdots,\zeta_j^n(p)$, defined on $\mathcal{U}_j$
such that for each $t$, $$\{p\rightarrow (\zeta_j^1(p),\cdots,\zeta_j^n(p))\ |\ \mathcal{U}_j\cap \pi^{-1}(t)\neq \emptyset\}$$ form a system of local holomorphic coordinates of $X_t$.
\end{enumerate}}
\end{definition}

Beltrami differential plays an important role in deformation theory.
A \emph{Beltrami differential} on $X$, generally
denoted by $\phi$, is an element in
$A^{0,1}(X, T^{1,0}_X)$, where $T^{1,0}_X$ is the holomorphic
tangent bundle of $X$. Then $\iota_\phi$ or $\phi\lrcorner$ denotes
the contraction operator with respect to $\phi\in A^{0,1}(X,T^{1,0}_X)$ or
other analogous vector-valued complex differential forms alternatively if there
is no confusion. We also use the convention
\begin{equation}\label{0e-convention}
e^{\spadesuit}=\sum_{k=0}^\infty \frac{1}{k!} \spadesuit^{k},
\end{equation}
where $\spadesuit^{k}$ denotes $k$-time action of the operator
$\spadesuit$. As the dimension of $X$ is finite, the summation in
the above formulation is always finite.

We will always consider the differentiable family $\pi:\mathcal{X} \rightarrow
B$ of compact complex $n$-dimensional manifolds
over a sufficiently small domain in $\mathbb{R}^k$ with the
reference fiber $X_0:= \pi^{-1}(0)$ and the general fibers $X_t:=
\pi^{-1}(t).$ For simplicity we set $k=1$. Denote by
$\zeta:=(\zeta^\alpha_j(z,t))$ the holomorphic coordinates of $X_t$
induced by the family with the holomorphic coordinates $z:=(z^i)$ of
$X_0$, under a coordinate covering $\{\mathcal{U}_j\}$ of
$\mathcal{X}$, when $t$ is assumed to be fixed, as the standard
notions in deformation theory described at the beginning of
\cite[Chapter 4]{MK}. This family induces a canonical differentiable family of
integrable Beltrami differentials on $X_0$, denoted by $\varphi(z,t)$,
$\varphi(t)$ and $\varphi$ interchangeably.

In \cite{RZ2,RZ}, the first and third authors introduced an extension
map
$$e^{\iota_{\varphi(t)}|\iota_{\overline{\varphi(t)}}}:
 A^{p,q}(X_0)\> A^{p,q}(X_t),$$ to play an important role in
 this paper.
\begin{definition}\label{map}\rm
For $s\in
 A^{p,q}(X_0)$, we define
$$\label{lbro} e^{\iota_{\varphi(t)}|\iota_{\overline{\varphi(t)}}}(s)=
s_{i_1\cdots i_pj_1\cdots
j_q}(z(\zeta))\left(e^{\iota_{\varphi(t)}}\left(dz^{i_1}\wedge\cdots\wedge
dz^{i_p}\right)\right)\wedge
\left(e^{\iota_{\overline{\varphi(t)}}}\left(d\overline{z}^{j_1}\wedge\cdots\wedge
d\overline{z}^{j_q}\right)\right),
$$
where $s$ is locally written as
$$s=s_{i_1\cdots i_pj_1\cdots
j_q}(z)dz^{i_1}\wedge\cdots\wedge dz^{i_p}\wedge
d\overline{z}^{j_1}\wedge\cdots\wedge d\overline{z}^{j_q}$$ and the
operators $e^{\iota_{\varphi(t)}}$,
$e^{\iota_{\overline{\varphi(t)}}}$ follow the convention
\eqref{0e-convention}. It is easy to check that this map is a real
linear isomorphism as in \cite[Lemma $2.8$]{RZ15}.
\end{definition}

The following proposition is crucial in this paper:
\begin{proposition}[{\cite[Theorem 3.4]{lry}, \cite[Proposition 2.2]{RZ15}}]\label{main1}
Let $\phi\in A^{0,1}(X,T^{1,0}_X)$ on a complex manifold $X$. Then on the space $A^{*,*}(X)$,
\beq\label{ext-old} d\circ e^{\iota_\phi}=e^{\iota_\phi}(d+\p\circ
\iota_\phi-\iota_\phi\circ
\p-\iota_{\db\phi-\frac{1}{2}[\phi,\phi]}).\eeq
\end{proposition}
From the proof of Proposition \ref{main1}, we see that \eqref{ext-old} is a natural generalization of Tian-Todorov Lemma\cite{T, To89}, whose variants appeared in
\cite{F,BK,Li,LSY,C} and also \cite{LR,lry} for vector bundle valued
forms.
\begin{lemma}\label{aaaa} For $\phi, \psi\in
A^{0,1}(X,T^{1,0}_X)$ and $\alpha\in A^{*,*}(X)$ on an $n$-dimensional complex
manifold $X$,
$$\label{f1}
[\phi,\psi]\lrcorner\alpha=-\p(\psi\lrcorner(\phi
\lrcorner\alpha))-\psi\lrcorner(\phi \lrcorner\p\alpha)
+\phi\lrcorner\p(\psi\lrcorner\alpha)+\psi
\lrcorner\p(\phi\lrcorner\alpha),
$$
where
$$[\phi,\psi]:=\sum_{i,j=1}^n(\phi^i\wedge\partial_i\psi^j+\psi^i\wedge\partial_i\phi^j)\otimes\partial_j$$ for
$\varphi=\sum_{i}\varphi^i\otimes\partial_i$
and $\psi=\sum_{i}\psi^i\otimes\partial_i$.
\end{lemma}

\subsection{The $p$-K\"{a}hler structures}
Let $V$ be a complex $n$-dimensional vector space with its
dual space $V^{*}$, i.e., the space of complex linear functionals
over $V$. Denote the complexified space of the exterior $m$-vectors
of $V^{*}$ by $\bigwedge^{m}_{\mathbb{C}} V^{*}$, which admits a
natural direct sum decomposition
\[ \bigwedge^{m}_{\mathbb{C}} V^{*} = \sum_{r+s=m} \bigwedge^{r,s} V^*, \]
where $\bigwedge^{r,s} V^*$ denotes the complex vector space
of $(r,s)$-forms on $V^*$.
The case $m=1$ exactly reads \[ \bigwedge^{1}_{\mathbb{C}} V^{*} =
V^* \bigoplus \overline{V^{*}},\] where the natural isomorphism $V^*
\cong \bigwedge^{1,0}V^*$ is used. Let $q\in \{1, \cdots, n\}$ and $p=n-q$. Clearly, the complex dimension
$N$ of $\bigwedge^{q,0}V^*$ equals to the combination number $C^{q}_n$. After a basis $\{
\beta_i \}_{i=1}^N$ of the complex vector space $\bigwedge^{q,0}V^*$ is fixed,
the canonical Pl\"ucker embedding as in \cite[Page 209]{GH} is given
by
$$\begin{array}{cccc}
\rho: & G(q,n) & \hookrightarrow & \mathbb{P}(\bigwedge^{q,0}V^*) \\
      & \Lambda & \mapsto & [\cdots,\Lambda_{i},\cdots]. \\
\end{array} $$
Here $G(q,n)$ denotes the Grassmannian of $q$-planes in the vector
space $V^*$ and $\mathbb{P}(\bigwedge^{q,0}V^*)$ is the
projectivization of $\bigwedge^{q,0}V^*$. A $q$-plane in $V^*$ can
be represented by a decomposable $(q,0)$-form $\Lambda \in
\bigwedge^{q,0}V^*$ up to a nonzero complex number, and
$\{\Lambda_i\}_{i=1}^N$ are exactly the coordinates of $\Lambda$
under the fixed basis $\{ \beta_i \}_{i=1}^N$. \emph{Decomposable
$(q,0)$-forms} are those forms in $ \bigwedge^{q,0}V^*$ that can be
expressed as $\gamma_1 \bigwedge \cdots \bigwedge \gamma_q$ with
$\gamma_i \in V^* \cong \bigwedge^{1,0}V^*$ for $1 \leq i \leq q$.
Set
$$\label{knpq}
k=(N-1)-pq
$$ to be the codimension of $\rho(G(q,n))$ in
$\mathbb{P}(\bigwedge^{q,0}V^*)$, whose locus characterizes the
decomposable $(q,0)$-forms in $\mathbb{P}(\bigwedge^{q,0}V^*)$.

Now we list several positivity notations and refer the readers to \cite{HK,H,Demailly} for more details. A $(q,q)$-form $\Theta$ in
$\bigwedge^{q,q}V^*$ is defined to be \emph{strictly positive (resp., positive)} if
\[ \Theta =\sigma_{q}\sum_{i,j=1}^N \Theta_{i\b j} \beta_i \wedge \b\beta_j,\]
where $\Theta_{ij}$ is a positive (resp., semi-positive) hermitian matrix of the size $N \times N$ with $N=C_{n}^q$ under the
basis $\{\beta_i \}_{i=1}^N$ of the complex vector space $\bigwedge^{q,0}V^*$ and $\sigma_{q}$ is defined to be the constant
$2^{-q}(\sqrt{-1})^{q^2}$.
According to this definition, the fundamental form of a hermitian metric on a complex manifold is actually a strictly positive $(1,1)$-form everywhere. A $(p,p)$-form $\Gamma\in \bigwedge^{p,p}V^*$ is called \emph{weakly positive} if
the volume form $$\Gamma\wedge\sigma_{q}\tau\wedge\bar{\tau}$$ is
positive for every nonzero decomposable $(q,0)$-form $\tau$ of $V^*$, while
a $(q,q)$-form $\Upsilon\in \bigwedge^{q,q}V^*$ is said to be \emph{strongly positive} if
$\Upsilon$ is a convex combination
$$\Upsilon=\sum_s\gamma_s \sqrt{-1}\alpha_{s,1}\wedge\bar\alpha_{s,1}\wedge\cdots\wedge\sqrt{-1}\alpha_{s,q}\wedge\bar\alpha_{s,q},$$
where $\alpha_{s,i}\in V^*$ and $\gamma_s\geq 0$.
As shown in \cite[Chapter III.\S\ 1.A]{Demailly}, the sets of weakly positive and strongly positive forms are closed convex cones,
and by definition, the weakly positive cone is dual to the strongly positive cone via the pairing
$$\bigwedge^{p,p}V^*\times \bigwedge^{q,q}V^*\longrightarrow \mathbb{C}.$$
Then all weakly positive forms are real.
An element $\Xi$ in $\bigwedge^{p,p}V^*$ is called \emph{transverse}, if
the volume form $$\Xi\wedge\sigma_{q}\tau\wedge\bar{\tau}$$ is
strictly positive for every nonzero decomposable $(q,0)$-form $\tau$
of $V^*$. There exist many various names for this
terminology and we refer to \cite[Appendix]{abb} for a list.

These positivity notations on complex vector spaces can be extended pointwise to
complex differential forms on a complex manifold.
Let $M$ be an $n$-dimensional  complex manifold. Then:

\begin{definition}[{\cite[Definition $1.11$]{aa}}, for example]\label{pkf}\rm
Let $p$ be an integer,
$1\leq p\leq n$. Then $M$ is called a \emph{$p$-K\"ahler manifold}
if there exists a \emph{$p$-K\"ahler form}, that is a $d$-closed transverse $(p,p)$-form on $M$.
\end{definition}

The readers are referred to \cite{S} for more related concepts (such as differential form transversal to the cone structure on a real differentiable manifold) to $p$-K\"ahler structures.

\section{Relevance to mild $\p\db$-lemma and modification}
We will introduce the so-called the $(p,q)$-th mild $\p\bar{\p}$-lemma and its relevance, and also present its modification stability on compact complex manifolds.
\subsection{The $(p,q)$-th mild $\p\db$-lemma and its relevance}\label{vddbar}
This subsection is to study various $\p\db$-lemmata related to local stabilities of complex structures, their properties, difference and roles there in some special case.
More details can be found in \cite[Subsection 3.1]{RwZ} and the references therein.

Now
we introduce a new notion.
\begin{definition}\label{mildddbar}\rm
  We say a complex manifold $X$ satisfies the \emph{$(p,q)$-th mild $\p\b{\p}$-lemma} if for any complex differential $(p-1,q)$-form $\xi$ with $\p\b{\p}\xi=0$ on $X$, there exists a $(p-1,q-1)$-form $\theta$ such that $\p\b{\p}\theta=\p\xi$.
\end{definition}

So we can state our main theorem:
\begin{theorem}\label{main}
For any positive integer $p\leq n-1$, any
small differentiable deformation $X_t$ of a
$p$-K\"ahler manifold $X_0$ satisfying  the
$(p,p+1)$-th mild $\p\db$-lemma is still
$p$-K\"ahlerian.
\end{theorem}

According to Lemma \ref{pkb}, Theorem \ref{main} unifies the local stabilities of K\"{a}hler structures \cite[Theorem 15]{KS} and balanced structures under $(n-1,n)$-th mild $\p\b{\p}$-lemma \cite[Theorem 1.5]{RwZ}, which is an obvious generalization of Wu's result \cite[Theorem 5.13]{w}
that the balanced structure is stable under small deformation when the $\p\db$-lemma holds, to $p$-K\"ahler mild $\p\db$-structures for $1\leq p \leq n-1$.

Let $X$ be a compact complex manifold of complex dimension $n$
with the following commutative diagram
\begin{equation}\label{diag}
\xymatrix{      & H^{p,q}_{\p}(X) \ar[dr]^{\iota^{p,q}_{\p,A}} &            \\
 H^{p,q}_{BC}(X) \ar[ur]^{\iota^{p,q}_{BC,\p}} \ar[dr]_{\iota^{p,q}_{BC,\db}} \ar[r]^{\iota^{p,q}_{BC,dR}} &  H^{p+q}_{dR}(X)\ar[r]^{\iota^{p,q}_{dR,A}} &  H^{p,q}_{A}(X)  \\
                & H^{p,q}_{\db}(X) \ar[ur]_{\iota^{p,q}_{\db,A}} &         . }
\end{equation}
Recall that Dolbeault cohomology groups $H^{\bullet,\bullet}_{\db}(X)$ of $X$ are defined by:
$$H^{\bullet,\bullet}_{\db}(X):=\frac{\ker\db}{\im\ \db},$$
with $H^{\bullet,\bullet}_{\p}(X)$ similarly defined, while Bott-Chern and Aeppli cohomology groups are defined as
$$H^{\bullet,\bullet}_{BC}(X):=\frac{\ker \p\cap \ker\db}{\im\ \p\db}\quad
\text{and}\quad H^{\bullet,\bullet}_{A}(X):=\frac{\ker \p\db}{\im\ \p+\im\ \db},$$
respectively. The dimensions of $H^{p+q}_{dR}(X)$, $H^{p,q}_{\db}(X)$, $H^{p,q}_{BC}(X)$, $H^{p,q}_{A}(X)$ and $H^{p,q}_{\p}(X)$ over $\mathbb{C}$
are denoted by $b_{p+q}(X)$, $h^{p,q}_{\db}(X)$, $h^{p,q}_{BC}(X)$, $h^{p,q}_{A}(X)$ and $h^{p,q}_{\p}(X)$, respectively, and the first four of them
are usually called $(p+q)$-th \emph{Betti numbers}, $(p,q)$-\emph{Hodge numbers}, \emph{Bott-Chern numbers} and \emph{Aeppli numbers}, respectively.
So the \emph{(standard) $\p\db$-lemma} is equivalent to the injectivities of the mappings
$$\iota^{p,q}_{BC,dR}: H^{p,q}_{BC}(X)\rightarrow H^{p+q}_{dR}(X)$$ for all $p,q$, or to the isomorphisms of all the maps in Diagram \eqref{diag} by \cite[Remark 5.16]{DGMS}.

Notice that the $(1,2)$-th mild $\p\db$-lemma is different from the $\p\db$-lemma on a complex manifold.
It is easy to see that the $(1,2)$-th mild $\p\db$-lemma amounts to the injectivity of the mapping
\[ \iota_{BC,\p}^{1,2}:H^{1,2}_{BC}(X) \rightarrow H^{1,2}_{\p}(X).\]
Then by \cite[Tables $5$ and $6$ in Appendix A]{AK} and \cite[the case $B$ in Example 1]{K}, one has:
\begin{example}[{\cite[Example 1.7]{RwZ}}]\label{exnak}
\emph{Let $X$ be the manifold in the case $(ii)$ of the completely-solvable
Nakamura manifold as given in \cite[Example 3.1]{AK}.
Then the manifold $X$ satisfies the $(1,2)$-th mild $\p\db$-lemma, but not the $\p\db$-lemma.}
\end{example}

There are another three similar conditions relating with the local
stabilities of complex structures. The \emph{$(p,p+1)$-th weak
$\p\db$-lemma} on a compact complex manifold $X$, first introduced by
Fu-Yau \cite{FY} in $(p,p+1)=(n-1,n)$, says that if for any real $(p,p)$-form $\psi$
such that $\db \psi$ is $\p$-exact, there is a $(p-1,p)$-form
$\theta$, satisfying
\[ \p \db \theta = \db \psi. \]
And the \emph{$(p,q)$-th strong $\p\db$-lemma} on $X$, first proposed by
Angella-Ugarte \cite{au} in the case $(p,q)=(n-1,n)$, states that the induced mapping
$\iota^{p,q}_{BC,A}: H^{p,q}_{BC}(X) \rightarrow
H^{p,q}_{A}(X)$ by the identity map is injective, which
is equivalent to that for any $d$-closed $(p,q)$-form $\Gamma$ of
the type $\Gamma=\p\xi+\db \psi$, there exists a $(p-1,q-1)$-form
$\theta$ such that
\[ \p \db \theta = \Gamma. \]
Angella-Ugarte \cite[Proposition 4.8]{au} showed the deformation
openness of the $(n-1,n)$-th strong $\p\db$-lemma.
Besides, the condition that the induced mapping
$\iota^{p,q}_{BC,\db}: H^{p,q}_{BC}(X) \rightarrow
H^{p,q}_{\db}(X)$ by the identity map is injective, is first presented
by Angella-Ugarte \cite{AU} in the case $(p,q)=(n-1,n)$ to study local conformal balanced
structures and global ones, which we may call the \emph{$(p,q)$-th
dual mild $\p\db$-lemma}.

After a simple check, we have the following observation:
\begin{observation}\label{s-eq}
\emph{The compact complex manifold $X$ satisfies the $(p,q)$-th
strong $\p\db$-lemma if and only if both of the mild and
dual mild ones hold on $X$.}
\end{observation}
All these four \lq\lq $\p\db$-lemmata" hold if the compact complex manifold $X$
satisfies the standard $\p\db$-lemma. And either the $(p,p+1)$-th mild or dual mild $\p\db$-lemma
implies the weak one, while \cite[Corollary 3.9]{RwZ} implies that
the $(n-1,n)$-th
mild $\p\db$-lemma and the dual mild one are unrelated.

By \cite{ab}, a small deformation of the
Iwasawa manifold, which satisfies the $(2,3)$-th weak $\p\db$-lemma
but does not satisfy the mild one from Example \ref{ex-dms-notms},
may not be balanced. Thus, the condition \lq\lq$(n-1,n)$-th mild
$\p\db$-lemma" in Corollary \ref{cor-mt}.\ref{balanced-Kahler} can't be replaced by the
weak one.
\begin{example}[{\cite[Example 3.7]{RwZ}}]\label{ex-dms-notms}
\emph{The complex structure in the category $(i)$ of
\cite[Proposition 2.3]{UV}, i.e., the complex parallelizable case of
complex dimension $3$, satisfies the $(2,3)$-th weak $\p\db$-lemma
and the dual mild one, but does not satisfy the mild one. The
Iwasawa manifold belongs to the category $(i)$.}
\end{example}

The next example shows that neither the $(n-1,n)$-th weak
$\p\db$-lemma nor the mild one is deformation open. And it shows
that the condition in \cite[Theorem 6]{FY} is not a necessary one
for the deformation openness of balanced structures as mentioned in
\cite[the discussion ahead of Example 3.7]{UV}. Recall that \cite[Theorem 6]{FY} says that the balanced
structure is deformation open, if the $(n-1,n)$-th weak
$\p\db$-lemma holds on the general fibers $X_t$ for $t \neq 0$. Fortunately,
Corollary \ref{cor-mt}.\ref{balanced-Kahler} can be applied to this example. See also
\cite[Remark 4.7]{au}, where Corollary \ref{cor-mt}.\ref{balanced-Kahler} can also be
applied.

\begin{example}[{\cite[Example 3.7]{UV}}]\label{not-fy}
\emph{Ugarte-Villacampa constructed an explicit family of
nilmanifolds with left-invariant complex structures $I_{\lambda}$ for
$\lambda\in[0,1)$ (of complex dimension $3$), with the fixed
underlying manifold the Iwasawa manifold. The complex structure of
the reference fiber $I_{0}$ is abelian and admits a left-invariant
balanced metric, satisfying the $(2,3)$-th mild $\p\db$-lemma by
\cite[Proposition 3.8]{RwZ}. The complex structures of $I_{\lambda}$
for $\lambda \neq 0$ are nilpotent from \cite[Cororllary 2]{CFP},
but neither complex-parallelizable nor abelian. Thus, they do not
satisfy the $(2,3)$-th weak $\p\db$-lemma by \cite[Proposition
3.6]{UV}. However, the nilmanifolds $I_{\lambda}$ for $\lambda \neq
0$ admit balanced metrics.}
\end{example}

Meanwhile, a $2n$-dimensional nilmanifold endowed with a left-invariant
abelian complex structure satisfies the $(n-1,n)$-th mild
$\p\db$-lemma but never satisfies the $(n-1,n)$-th dual mild
$\p\db$-lemma. It shows that the deformation openness of balanced
structures with the reference fiber a nilmanifold
endowed with a left-invariant abelian balanced Hermitian structure is easily obtained by Corollary \ref{cor-mt}.\ref{balanced-Kahler}, but not
from \cite[Theorem 4.9]{au}, which says that if $X_0$ admits a locally conformal balanced
metric and satisfies the $(n-1,n)$-th strong $\p\db$-lemma, then $X_t$ is balanced for small $t$.

Moreover, the deformation invariance of the $(n-1,n-1)$-th Bott-Chern numbers
$h^{n-1,n-1}_{BC}(X_t)$ can assure the deformation openness of
balanced structures as shown in \cite[Proposition 4.1]{au}.
Inspired by Wu's result \cite[Theorem 5.13]{w}, one has the generalization:
\begin{theorem}[{\cite[Theorem $1.9$]{RwZ}}]\label{d-ext}
For any positive integer $p\leq n-1$, any
small differentiable deformation $X_t$ of a
$p$-K\"ahler manifold $X_0$ satisfying the deformation invariance of $(p,p)$-Bott-Chern numbers
is still $p$-K\"ahlerian.
\end{theorem}

Nevertheless, Corollary \ref{cor-mt}.\ref{balanced-Kahler} may be applied to some cases with deformation variance
of the $(n-1,n-1)$-th Bott-Chern numbers. The coming newly constructed example can be one of them among nilmanifolds,
while the manifold in \cite[Example 4.10]{au}, satisfying the $(2,3)$-th strong $\p\db$-lemma, is a solvable manifold
but not a nilmanifold by \cite[Corollary 3.9]{RwZ}.
\begin{example}\label{bcvary}
\emph{Let $G$ be the simply connected nilpotent Lie group
determined by a ten-dimensional $3$-step nilpotent Lie algebra $\mathfrak{g}$
endowed with a left-invariant abelian complex structure $J$, satisfying the following structure equation:
\[\begin{cases}
d\gamma^1=d\gamma^2=d\gamma^3=0,\\
d\gamma^4=\gamma^{1\bar{3}},\\
d\gamma^5=\gamma^{3\bar{4}},
\end{cases} \]
where the natural decomposition with respect to $J$ yields
\[ \mathfrak{g}_{\mathbb{C}} = \mathfrak{g} \otimes_{\mathbb{R}} \mathbb{C} = \mathfrak{g}^{1,0}_J \oplus \mathfrak{g}^{0,1}_J;
\mathfrak{g}^*_{\mathbb{C}} = \mathfrak{g}^* \otimes_{\mathbb{R}} \mathbb{C} = \mathfrak{g}^{*(1,0)}_J \oplus \mathfrak{g}^{*(0,1)}_J, \]
$\{ \gamma^i \}_{i=1}^5$ is the basis of $\mathfrak{g}^{*(1,0)}$
and the convention $\gamma^{1\bar{3}}=\gamma^{1} \wedge \overline{\gamma^{3}}$ is
used here and afterwards. Define a lattice $\Gamma$ in $G$, determined by the rational
span of $\{ \gamma^i,\bar{\gamma}^i \}_{i=1}^5$. Then $M := \Gamma \backslash G$ is a compact nilmanifold with the abelian complex structure $J$ given above. It is easy to check that
\[\Omega=\gamma^{1234\overline{1234}}+\gamma^{1235\overline{1235}}
+\gamma^{1245\overline{1245}}+\gamma^{1345\overline{1345}}+\gamma^{2345\overline{2345}}\]
is a left-invariant balanced metric on $(\mathfrak{g},J)$, descending to $M$.
Denote the basis of $\mathfrak{g}^{1,0}$ dual to $\{ \gamma^i \}_{i=1}^5$ by $\{\theta_i\}_{i=1}^5$.
The equation $d\omega(\theta,\theta') = - \omega([\theta,\theta'])$ for $\omega \in \mathfrak{g}^{*}_{\mathbb{C}}$
and $\theta,\theta' \in \mathfrak{g}_{\mathbb{C}}$, establishes the equalities
\[ [\bar{\theta}_3,\theta_1]=\theta_4,\ [\bar{\theta}_4,\theta_3]=\theta_5. \]
According to \cite[Theorem 3.6]{CFP}, the linear operator $\db$ on $\mathfrak{g}^{1,0}$, defined in \cite[Section 3.2]{CFP},
amounts to
\[ \db: \mathfrak{g}^{1,0} \rightarrow \mathfrak{g}^{*(0,1)} \otimes \mathfrak{g}^{1,0}:
\ \db V = \bar{\gamma}^i \otimes [\bar{\theta}_i,V]^{1,0}\ \ \text{for}\ V \in \mathfrak{g}^{1,0},\]
which induces an isomorphism $H^1(M,T^{1,0}_{M}) \cong H^1_{\db}(\mathfrak{g}^{1,0})$.
Therefore, from Kodaira-Spencer's deformation theory,
an analytic deformation $M_t$ of $M$ can be constructed by use of the integrable left-invariant Beltrami differential
\[ \varphi(t) = (t_1 \bar{\gamma}^4 + t_2 \bar{\gamma}^5) \otimes \theta_2 + (t_3 \bar{\gamma}^4 + t_4 \bar{\gamma}^5 ) \otimes \theta_5 \]
for $t=(t_1,t_2,t_3,t_4)$ and $|t_4|<1$, which satisfies $\db \varphi(t) = \frac{1}{2}[\varphi(t),\varphi(t)]$
and the so-called Schouten-Nijenhuis bracket $[\cdot,\cdot]$ (cf. \cite[Formula (4.1)]{CFP}) works as
\[ [\bar{\gamma} \otimes \theta, \bar{\gamma}' \otimes \theta']
=  \bar{\gamma}' \w \iota_{\theta'} d\bar{\gamma} \otimes \theta + \bar{\gamma} \w \iota_{\theta} d\bar{\gamma}' \otimes \theta'\ \ \text{for}\ \gamma,\gamma' \in \mathfrak{g}^{*(1,0)}, \theta,\theta' \in \mathfrak{g}^{1,0}.\]
Then the general fibers $M_t$ are still nilmanifolds, determined by the Lie algebra $\mathfrak{g}$ with respect to the decompositions
\[ \mathfrak{g}_{\mathbb{C}} = \mathfrak{g} \otimes_{\mathbb{R}} \mathbb{C} = \mathfrak{g}^{1,0}_{\varphi(t)} \oplus \mathfrak{g}^{0,1}_{\varphi(t)}; \mathfrak{g}^*_{\mathbb{C}} = \mathfrak{g}^* \otimes_{\mathbb{R}} \mathbb{C} = \mathfrak{g}^{*(1,0)}_{\varphi(t)} \oplus \mathfrak{g}^{*(0,1)}_{\varphi(t)}, \]
where the basis of $\mathfrak{g}^{*(1,0)}_{\varphi(t)}$ is given by
$\gamma^i(t) = e^{\iota_{\varphi(t)}} \big( \gamma^i \big) = \big(\1 + \varphi(t)\big) \lrcorner \gamma^i$ for $1 \leq i \leq 5$.
Hence, the structure equation of $\{ \gamma^i(t) \}_{i=1}^5$ is
\[\begin{cases}
d\gamma^1(t)=d\gamma^3(t) = 0, \\
d\gamma^2(t)=-t_1 \gamma^{3\bar{1}}(t)-t_2\gamma^{4\bar{3}}(t), \\
d\gamma^4(t)=\gamma^{1\bar{3}}(t),\\
d\gamma^5(t)=\gamma^{3\bar{4}}(t)-t_3\gamma^{3\bar{1}}(t)-t_4\gamma^{4\bar{3}}(t),\\
\end{cases}\]
where $\gamma^{3\bar{1}}(t)$ denotes $\gamma^{3}(t) \wedge \overline{\gamma^{1}(t)}$, similarly for others.
It is well known from \cite{CF,R,A} that the Bott-Chern cohomologies of nilmanifolds with abelian complex structures
and their small deformation can be calculated via left-invariant differential forms. Remark \ref{n-1-d-clsd-ext} tells
us that the dimension of the space of the $d$-closed left-invariant $(4,4)$-forms is invariant along the deformation $M_t$, which
is equal to $21$. And one can calculate the $\p\db$-exact terms directly by use of the structure equation:
\[\begin{aligned}
& \p_t \db_t \left( \mathfrak{g}^{*(3,3)}_{\varphi(t)}\right)\\
=& \left\langle -(1+|t_4|^2)\gamma^{1234\overline{1234}}(t)
- |t_2|^2 \gamma^{1345\overline{1345}}
+ t_2\overline{t}_4 \gamma^{1345\overline{1234}}(t) + t_4 \overline{t}_2 \gamma^{1234\overline{1345}}(t),\right.\\
& -\gamma^{1235\overline{1235}}(t)-|t_1|^2\gamma^{1345\overline{1345}}(t)-|t_3|^2\gamma^{1234\overline{1234}}(t)
+t_1\overline{t}_3\gamma^{1345\overline{1234}}(t)+t_3\overline{t}_1\gamma^{1234\overline{1345}}(t),\\
& -t_1\gamma^{1345\overline{1234}}(t) + t_2 \gamma^{1345\overline{1235}}(t)
+t_3 \gamma^{1234\overline{1234}}(t)-t_4\gamma^{1234\overline{1235}}(t), \\
& \left. -\overline{t}_1\gamma^{1234\overline{1345}}(t) + \overline{t}_2 \gamma^{1235\overline{1345}}(t)
+\overline{t}_3 \gamma^{1234\overline{1234}}(t)-\overline{t}_4\gamma^{1235\overline{1234}}(t) \right\rangle.
\end{aligned}\]
It is clear that $\dim\p \db \left( \mathfrak{g}^{*(3,3)}_{J} \right)=2$
and $\dim\p_t \db_t\left( \mathfrak{g}^{*(3,3)}_{\varphi(t)}\right)=4$ for general $t$.
Therefore, the Bott-Chern number $h^{4,4}_{BC}(M_t)$ varies from $19$ to $17$ along the deformation $M_t$.}
\end{example}

It may not be difficult to find an example of a non-K\"ahler $p$-K\"ahler manifold in the Fujiki class for $1<p<n-1$ in the literature,
and thus it satisfies the $(p,p+1)$-th mild $\p\db$-lemma, for example \cite[Section 4]{aabb}. However, motivated by Example \ref{bcvary}, we try to ask:
\begin{question}
\emph{Is it possible to find an $n$-dimensional nilmanifold with a left-invariant complex structure of complex dimension $n$,
which admits a left-invariant $p$-K\"ahler metric for $1<p<n-1$ and satisfies the
$(p,p+1)$-th mild $\p\db$-lemma, but the $(p,p)$-Bott Chern number varies along some deformation?
This example would not satisfy the standard $\p\db$-lemma.}
\end{question}

Finally, from the perspective of Corollary \ref{cor-mt}.\ref{balanced-Kahler}, we may have
a clear picture of Angella-Ugarte's result \cite[Theorem 4.9]{au}. Actually, Observation \ref{s-eq} tells us that the $(n-1,n)$-th strong
$\p\db$-lemma decomposes into the mild one and the dual mild one. A locally conformal balanced
metric can be transformed into a balanced one by the $(n-1,n)$-th
dual mild $\p\db$-lemma, from \cite[Theorem 2.5]{AU}. Then the
$(n-1,n)$-th mild $\p\db$-lemma assures the deformation
openness of balanced structures originally from the transformed balanced
metric on the reference fiber, thanks to Corollary \ref{cor-mt}.\ref{balanced-Kahler}.

\subsection{Modification stabilities}
We consider the modification on a compact complex manifold defined as follows and refer the reader to \cite[$\S$ 2]{u} for
its general definition on complex spaces.
\begin{definition}\rm
A \emph{modification} of an $n$-dimensional compact complex manifold $ M $ is a holomorphic map
$$\mu: \tilde{M} \rightarrow  M $$
so that:
\begin{enumerate}[$(i)$]
    \item
$ \tilde{M} $ is also an $n$-dimensional compact complex manifold;
    \item
There exists an analytic subset $S\subseteq M$ of codimension greater than or equal to $1$ such that $\mu\mid_{\tilde{M}\setminus\mu^{-1}(S)}:\tilde{M}\setminus\mu^{-1}(S)\rightarrow M\setminus S$ is a biholomorphism.
%
 \end{enumerate}

\end{definition}

It is a classical result \cite{par} or \cite[Theorem 5.22]{DGMS} that if the modification of a complex manifold is a {$\partial\bar\partial$-manifold}, then so is this manifold. Therefore, each compact complex manifold
in the \emph{Fujiki class $\mathcal{C}$} (i.e. admitting a K\"ahler modification) is a $\partial\bar\partial$-manifold.
The converse is an open question as in \cite[Introduction]{a}: Is the modification of a $\partial\bar\partial$-manifold still a $\partial\bar\partial$-manifold?
A recent result \cite[Theorem 1.3]{YY17} of S. Yang and X. Yang, by means of a blow-up formula for Bott-Chern cohomologies and the characterizations by Angella-Tomassini \cite{at} and Angella-Tardini \cite{atr} of $\partial\bar\partial$-manifolds, and also \cite[Main Theorem 1.1]{ryy} by S. Yang, X. Yang and the first author confirm this question in three dimension, that is, the modification of a $\partial\bar\partial$-threefold is still a $\partial\bar\partial$-threefold. See also more recent \cite[Theorem 2.1]{astt}. These results provide us more classes of complex manifolds satisfying mild $\partial\bar\partial$-lemmata. Moreover, it is natural to ask the analogous:
\begin{question}
\emph{Does the modification of a complex manifold satisfying the mild $(p,q)$-th $\partial\bar\partial$-lemma still satisfy the mild $(p,q)$-th $\partial\bar\partial$-lemma for each $p,q$?
}\end{question}

Now we present a modification stability of $(p,q)$-th mild $\p\db$-lemma on a compact complex manifold.
Let $ M $ be a complex manifold. One has the $\mathbb{K}$-valued de Rham complex $(A^\bullet(M)_{\mathbb{K}},d)$ for $\mathbb{K}\in \{\mathbb{R},\mathbb{C}\}$. Fixing a $d$-closed $1$-form $\phi\in A^1(M)_\mathbb{K}$, we consider another complex. Namely, define
$$d_\phi=d+L_\phi,$$
where $L_\phi:=\phi\wedge\bullet$. The cochain complex
$$(A^\bullet(M)_{\mathbb{K}},d_\phi)$$
can be regarded as the de Rham complex with values in the topologically trivial flat bundle $M\times \mathbb{K}$ with the connection form $\phi$.
We study cohomologies and Hodge theory for general complex manifolds with twisted differentials. More precisely, for $\theta_1,\theta_2\in H^{1,0}_{BC}(M)$, consider the bi-differential $\mathbb{Z}$-graded complex
$$(A^\bullet(M)_{\mathbb{C}},\p_{(\theta_1,\theta_2)},\db_{(\theta_1,\theta_2)}),$$
where $$\p_{(\theta_1,\theta_2)}:=\p+L_{\theta_2}+L_{\overline{\theta_1}},$$
$$\db_{(\theta_1,\theta_2)}:=\db-L_{\overline{\theta_2}}+L_{\theta_1}.$$
It is easy to check that
$$\p_{(\theta_1,\theta_2)}\p_{(\theta_1,\theta_2)}=\db_{(\theta_1,\theta_2)}\db_{(\theta_1,\theta_2)}
=\p_{(\theta_1,\theta_2)}\db_{(\theta_1,\theta_2)}+\db_{{(\theta_1,\theta_2)}}{\p_{(\theta_1,\theta_2)}}=0.$$
Angella and Kasuya \cite{AK1} investigated cohomological properties of this bi-differential complex and considered (more than) two cohomologies:
$$ H^\bullet(A^\bullet(M)_{\mathbb{C}};\p_{(\theta_1,\theta_2)},\db_{(\theta_1,\theta_2)};\p_{(\theta_1,\theta_2)}\db_{(\theta_1,\theta_2)}):= \frac{\ker \p_{(\theta_1,\theta_2)}\cap\ker \db_{(\theta_1,\theta_2)}}{\im\ \p_{(\theta_1,\theta_2)}\db_{(\theta_1,\theta_2)}}$$
and
$$ H^\bullet(A^\bullet(M)_{\mathbb{C}};\p_{(\theta_1,\theta_2)}):= \frac{\ker \p_{(\theta_1,\theta_2)}}{\im\ \p_{(\theta_1,\theta_2)}},$$
which are simply denoted by $H^\bullet_{BC}(M,\theta_1,\theta_2)$ and $H^\bullet_{\p}(M,\theta_1,\theta_2)$, respectively. If one sets $\theta_1=\theta_2=0$,
$H^\bullet_{BC}(M,\theta_1,\theta_2)$ and $H^\bullet_{\p}(M,\theta_1,\theta_2)$ are just the ordinary {Bott-Chern cohomology} $H^{\bullet,\bullet}_{BC}(M)$ and $H^{\bullet,\bullet}_{\p}(M)$ of $ M $, respectively.

Following R. O. Wells in \cite[Theorem 3.1]{Wells74},
Angella and Kasuya proved:
\begin{proposition}[{\cite[Theorem 2.4]{AK1}}]\label{bca-inj}
Let $\mu: \tilde{M} \rightarrow  M $ be a modification of a compact complex manifold $ M $. Then the induced maps
$$\mu^*_{BC}:H^{\bullet}_{BC}(M,\theta_1,\theta_2)\rightarrow H^{\bullet}_{BC}(\tilde{M},\mu^*\theta_1,\mu^*\theta_2),$$
$$\mu^*_{\p}:H^{\bullet}_{\p}(M,\theta_1,\theta_2)\rightarrow H^{\bullet}_{\p}(\tilde{M},\mu^*\theta_1,\mu^*\theta_2)$$
are injective.
\end{proposition}
We reformulate Definition \ref{mildddbar} for the $(p,q)$-th mild $\p\b{\p}$-lemma as follows.
\begin{definition}\rm For any positive integers $p,q\leq n$, an $n$-dimensional complex manifold $X$ satisfies the \emph{$(p,q)$-th mild $\p\b{\p}$-lemma} if the induced map $\iota_{BC,\p}^{p,q}: H^{p,q}_{BC}(X) \rightarrow H^{p,q}_{\p}(X)$  by the identity map is injective.
\end{definition}

\begin{proposition}\label{mefm}
With the notations in Proposition \ref{bca-inj}, if $ \tilde{M} $ satisfies the $(p,q)$-th mild $\p\b{\p}$-lemma, then so does $ M $.
\end{proposition}
\begin{proof} Taking $\theta_1=\theta_2=0$, one has the commutative diagram:
$$\xymatrix{H^{p,q}_{BC}(M) \ar[r]^{\mu^*_{BC}}\ar[d]_{\iota_{BC,\p}^{p,q}}& H^{p,q}_{BC}(\tilde{M})\ar[d]^{\iota_{BC,\p}^{p,q}}\\
H^{p,q}_{\p}(M)\ar[r]^{\mu^*_{\p}}  & H^{p,q}_{\p}(\tilde{M}).}$$
By the $(p,q)$-th mild $\p\b{\p}$-lemma assumption on $\tilde{M}$, the map $\iota_{BC,\p}^{p,q}$ for $\tilde{M}$ is injective and so are $\mu^*_{BC},\mu^*_{\p}$ by Proposition \ref{bca-inj}. So the map $\iota_{BC,\p}^{p,q}$ for ${M}$ is injective, i.e., $ M $ satisfies the $(p,q)$-th mild $\p\b{\p}$-lemma.
\end{proof}

\section{Power series proof of main result}\label{psp}
This section is used to prove main Theorem \ref{main}.
Let us sketch Kodaira-Spencer's proof of local stability theorem
\cite{KS}. Let $F_t$ be the orthogonal projection to the kernel
$\mathds{F}_t$ of the \emph{first $4$-th order Kodaira-Spencer
operator} (also often called \emph{Bott-Chern Laplacian})
\begin{equation}\label{bc-Lap}
\square_{{BC},t}=\p_t\db_t\db_t^*\p_t^*+\db_t^*\p_t^*\p_t\db_t+\db_t^*\p_t\p_t^*\db_t+\p_t^*\db_t\db_t^*\p_t+\db_t^*\db_t+\p_t^*\p_t
\end{equation}
and $\G_t$ the corresponding Green's operator with respect to $\alpha_t$ on $X_t$.
Here
$$\alpha_t=\sqrt{-1}g_{i\bar{j}}(\zeta,t)d \zeta^i \w d \overline{\zeta}^j$$
is a hermitian metric on $X_t$ depending differentiably on $t$ with $\alpha_0$ being a K\"ahler metric on $X_0$,
and $\db_t^*$ (resp., $\p_t^*$) is the dual of $\db_t$ (resp., $\p_t$) with respect to $\alpha_t$. By a cohomological
argument with the upper semi-continuity theorem, they prove that
$F_t$ and $\G_t$ depend differentiably on $t$. Then they can
construct the desired K\"ahler metric on $X_t$ as
$$\widetilde{\alpha_t}=\frac{1}{2}(F_t\alpha_t+\overline{F_t\alpha_t}).$$
See also \cite[Subsection 9.3]{V}.

Our proof is quite different.
As explained in Section \ref{intro}, to prove Theorem \ref{main}, it suffices to prove the special case $p=q$ of:
  \begin{theorem}\label{pq-sm-ext}
  If $X_0$ satisfies the $(p,q+1)$- and $(q,p+1)$-th mild $\p\b{\p}$-lemmata, then there is a  $d$-closed $(p,q)$-form $\Omega(t)$ on $X_t$ depending smoothly on $t$ with $\Omega(0)=\Omega_0$ for any  $d$-closed $\Omega_0\in A^{p,q}(X_0)$.
  \end{theorem}

We first reduce the local stability Theorem \ref{pq-sm-ext} to the Kuranishi family since the family of Beltrami differentials induced by this Kuranishi family plays an important role in the construction of the family of $d$-closed $(p,q)$-forms $\Omega(t)$.

\subsection{Kuranishi family and Beltrami differentials}
We introduce some basics on Kuranishi family of complex structures in this subsection, which is extracted from \cite{RZ15,RwZ} and
obviously originally from \cite{ku}.

By (the proof of) Kuranishi's completeness theorem \cite{ku}, for any compact complex manifold $X_0$, there exists a complete holomorphic family
$\varpi:\mc{K}\to T$ of complex manifolds at the reference point $0\in T$ in the sense that for any differentiable family $\pi:\mc{X}\to B$ with $\pi^{-1}(s_0)=\varpi^{-1}(0)=X_0$, there exist a sufficiently small neighborhood $E\subseteq B$ of $s_0$, and smooth maps $\Phi: \mathcal {X}_E\rightarrow \mathcal {K}$,  $\tau: E\rightarrow T$ with $\tau(s_0)=0$ such that the diagram commutes
$$\xymatrix{\mathcal {X}_E \ar[r]^{\Phi}\ar[d]_\pi& \mathcal {K}\ar[d]^\varpi\\
(E,s_0)\ar[r]^{\tau}  & (T,0),}$$
$\Phi$ maps $\pi^{-1}(s)$ biholomorphically onto $\varpi^{-1}(\tau(s))$ for each $s\in E$, and $$\Phi: \pi^{-1}(s_0)=X_0\rightarrow \varpi^{-1}(0)=X_0$$ is the identity map.
This family is called \emph{Kuranishi family} and constructed as follows. Let $\{\eta_\nu\}_{\nu=1}^m$ be a base for $\mathbb{H}^{0,1}(X_0,T^{1,0}_{X_0})$, where some suitable hermitian metric is fixed on $X_0$ and $m\geq 1$; Otherwise the complex manifold $X_0$ would be \emph{rigid}, i.e., for any differentiable family $\kappa:\mc{M}\to P$ with $s_0\in P$ and $\kappa^{-1}(s_0)=X_0$, there is a neighborhood $V \subseteq P$ of $s_0$ such that $\kappa:\kappa^{-1}(V)\to V$ is trivial. Then one can construct a holomorphic family
$$\label{phi-ps-pp}\varphi(t) = \sum_{|I|=1}^{\infty}\varphi_{I}t^I:=\sum_{j=1}^{\infty}\varphi_j(t),\ I=(i_1,\cdots,i_m),\ t=(t_1,\cdots,t_m)\in \mathbb{C}^m,$$ {for $|t|< \rho$ a small positive constant,} of Beltrami differentials as follows:
$$\label{phi-ps-0}
 \varphi_1(t)=\sum_{\nu=1}^{m}t_\nu\eta_\nu
$$
and for $|I|\geq 2$,
$$\label{phi-ps}
  \varphi_I=\frac{1}{2}\db^*\G\sum_{J+L=I}[\varphi_J,\varphi_{L}].
$$
It is clear that $\varphi(t)$ satisfies the equation
$$\varphi(t)=\varphi_1+\frac{1}{2}\db^*\G[\varphi(t),\varphi(t)].$$
Let
$$T=\{t\ |\ \mathbb{H}[\varphi(t),\varphi(t)]=0 \}.$$
So for each $t\in T$, $\varphi(t)$ satisfies
$$\label{int}
\b{\p}\varphi(t)=\frac{1}{2}[\varphi(t),\varphi(t)],
$$
and determines a complex structure $X_t$ on the underlying differentiable manifold of $X_0$. More importantly, $\varphi(t)$ represents the complete holomorphic family $\varpi:\mc{K}\to T$ of complex manifolds. Roughly speaking, Kuranishi family $\varpi:\mc{K}\to T$ contains all sufficiently small differentiable deformations of $X_0$.

By means of these, one can reduce the local stability Theorem \ref{pq-sm-ext} to the Kuranishi family by shrinking $E$ if necessary, that is, it suffices to construct a $p$-K\"{a}hler metric on each $X_t$. From now on, we use $\varphi(t)$ and  $\varphi$ interchangeably to denote this holomorphic family of integrable Beltrami differentials, and assumes $m=1$ for simplicity.

\subsection{Obstruction equation and construction of power series}
Now we begin to prove the $d$-closed smooth extension of $(p,q)$-forms as in Theorem \ref{pq-sm-ext} by using power series method.

As both $e^{\iota_{(\1-\b{\varphi}\varphi)^{-1}\b{\varphi}}}$ and
$e^{\iota_{\varphi}}$ are invertible operators when $t$ is
sufficiently small, it follows that for any $\Omega \in
A^{p,q}(X_0)$,
\begin{align}\label{2.1.1}
  e^{\iota_{\varphi}|\iota_{\b{\varphi}}}(\Omega)=e^{\iota_{\varphi}}\circ e^{\iota_{(\1-\b{\varphi}\varphi)^{-1}\b{\varphi}}}\circ e^{-\iota_{(\1-\b{\varphi}\varphi)^{-1}\b{\varphi}}}\circ e^{-\iota_{\varphi}}\circ e^{\iota_{\varphi}|\iota_{\b{\varphi}}}(\Omega).
\end{align}
Set
\begin{align}\label{2.1.2}
  \tilde{\Omega}=e^{-\iota_{(\1-\b{\varphi}\varphi)^{-1}\b{\varphi}}}\circ e^{-\iota_{\varphi}}\circ e^{\iota_{\varphi}|\iota_{\b{\varphi}}}(\Omega),
\end{align}
where $\Omega$ and $\tilde{\Omega}$ are apparently one-to-one
correspondence. Here we follow the notations:
$\overline{\varphi} \varphi = \varphi \lc \overline{\varphi}$,
$\1$ is the identity operator defined as:
$$\1=\frac{1}{p+q}\left(\sum_i^n dz^i\otimes \frac{\p\ }{\p z^i}+\sum_i^n d\bar z^i\otimes \frac{\p\ }{\p \bar z^i}\right)$$
when it acts on $(p,q)$-forms of a complex manifold, and similarly for others.
 And it is easy to check that the operator
$$e^{-\iota_{(\1-\b{\varphi}\varphi)^{-1}\b{\varphi}}}\circ
e^{-\iota_{\varphi}}\circ e^{\iota_{\varphi}|\iota_{\b{\varphi}}}$$
preserves the form types and thus $\tilde{\Omega}$ is still a $(p,q)$-form. In fact, for any $(p,q)$-form $\alpha$ on $X_0$,
we will find
\begin{align*}
\begin{split}
&\ e^{-\iota_{(\1-\b{\varphi}\varphi)^{-1}\b{\varphi}}}\circ
e^{-\iota_{\varphi}}\circ e^{\iota_{\varphi}|\iota_{\b{\varphi}}}(\alpha) \\
=&\ \alpha_{i_1\cdots i_p\o{j_1}\cdots\o{j_q}}
dz^{i_1}\wedge\cdots\wedge dz^{i_p} \wedge(\1-\b{\varphi}\varphi)\l
d\b{z}^{j_1}\wedge \cdots\wedge (\1-\b{\varphi}\varphi)\l
d\b{z}^{j_q}\\
\in&\ A^{p,q}(X_0),
\end{split}
\end{align*}
where $\alpha=\alpha_{i_1\cdots
i_p\o{j_1}\cdots\o{j_q}}dz^{i_1}\wedge\cdots\wedge dz^{i_p}\wedge
d\b{z}^{j_1}\wedge \cdots\wedge d\b{z}^{j_q}$. Together with
\eqref{2.1.1} and \eqref{2.1.2}, Proposition \ref{main1} implies that
\begin{align}\label{2.2.1}
\begin{split}
  d(e^{\iota_{\varphi}|\iota_{\b{\varphi}}}(\Omega))&=d\circ e^{\iota_{\varphi}}\circ e^{\iota_{(\1-\b{\varphi}\varphi)^{-1}\b{\varphi}}}(\tilde{\Omega}) \\
                                                    &=e^{\iota_{\varphi}}\circ(\b{\p}+[\p, \iota_{\varphi}]+\p)\circ e^{\iota_{(\1-\b{\varphi}\varphi)^{-1}\b{\varphi}}}(\tilde{\Omega}) \\
                                                    &=e^{\iota_{\varphi}}(\b{\p}_{\varphi}+\p)\sum_{k=0}^{+\infty}A_k\\
                                                    &=e^{\iota_{\varphi}}\left(\b{\p}_{\varphi}A_0 +\sum_{k=0}^{+\infty}(\p A_k+\b{\p}_{\varphi}A_{k+1})\right),
\end{split}
\end{align}
where $$A_k:=\frac{(\iota_{(\1-\b{\varphi}\varphi)^{-1}\b{\varphi}})^k}{k!}(\tilde{\Omega})$$  is a $(p+k,q-k)$-form and $$\b{\p}_{\varphi}:=\b{\p}+[\p,\iota_\varphi].$$
Thus, $d(e^{\iota_{\varphi}|\iota_{\b{\varphi}}}(\Omega))=0$ amounts to
$$\label{obstruction}
  \b{\p}_{\varphi}A_0=0,\quad \p A_k+\b{\p}_{\varphi}A_{k+1}=0,\quad k=0,1,2,\ldots.
$$

\begin{proposition} \label{d-cri}
For any given $(p,q)$-form $\Omega$,
  $$d(e^{\iota_{\varphi}|\iota_{\b{\varphi}}}(\Omega))=0$$ is equivalent to
\begin{align}\label{1.4}
  \sum_{k=0}^\infty\left(\b{\p}\circ\frac{\iota^{k-1}_{\varphi}}{(k-1)!}+\p\circ\frac{\iota^{k}_{\varphi}}{k!}\right)A_{n-p-(l-k)}=0,
\end{align}
where $\max\{1,n-p-q\}\leq l\leq\min\{2n-p-q,n+1\}$, $\iota^{k}_{\varphi}=0$ for $k<0$ and $0!=1$.
\end{proposition}
\begin{proof}
  Recall that $\tilde{\Omega}=e^{-\iota_{(\1-\b{\varphi}\varphi)^{-1}\b{\varphi}}}\circ e^{-\iota_{\varphi}}\circ e^{\iota_{\varphi}|\iota_{\b{\varphi}}}(\Omega)$ and then
  \begin{align*}
    d(e^{\iota_{\varphi}|\iota_{\b{\varphi}}}(\Omega))&=d\circ e^{\iota_{\varphi}}\circ e^{\iota_{(\1-\b{\varphi}\varphi)^{-1}\b{\varphi}}}(\tilde{\Omega})\\
    &=(\b{\p}\circ e^{\iota_{\varphi}}\circ e^{\iota_{(\1-\b{\varphi}\varphi)^{-1}\b{\varphi}}}+\p\circ e^{\iota_{\varphi}}\circ e^{\iota_{(\1-\b{\varphi}\varphi)^{-1}\b{\varphi}}})(\tilde{\Omega})\\
    &=\sum_{k_1,k_2=0}^{+\infty}\left(\b{\p}\circ\frac{\iota^{k_1}_{\varphi}}{k_1!}
    +\p\circ\frac{\iota^{k_1}_{\varphi}}{k_1!}\right)\circ\frac{\iota^{k_2}_{(\1-\b{\varphi}\varphi)^{-1}\b{\varphi}}}{k_2!}(\tilde{\Omega}).
  \end{align*}
  Note that the part of degree $(+(n-p-l+1),-(n-p-l))$ in the operator $d\circ e^{\iota_{\varphi}}\circ e^{\iota_{(\1-\b{\varphi}\varphi)^{-1}\b{\varphi}}}$ is
  \begin{align}\label{1.5}
    \sum_{k=0}^{\infty}\left(\b{\p}\circ\frac{\iota^{k-1}_{\varphi}}{(k-1)!}+\p\circ\frac{\iota^{k}_{\varphi}}{k!}\right)\circ\frac{\iota^{n-p-l+k}_{(\1-\b{\varphi}\varphi)^{-1}\b{\varphi}}}{(n-p-l+k)!}(\tilde{\Omega})
  \end{align}
  since \begin{align*}
  \ &(+(n-p-l+1),-(n-p-l))\\
 =\ &(n-p-l+k)(1,-1)+(k-1)(-1,1)+(0,1)\\
 =\ &(n-p-l+k)(1,-1)+k(-1,1)+(1,0).
\end{align*}
  This is exactly the left-hand side of (\ref{1.4}). So $d(e^{\iota_{\varphi}|\iota_{\b{\varphi}}}(\Omega))=0$ is equivalent to the vanishing of (\ref{1.5}) for each $l$ such that
$$
  (p,q)+(+(n-p-l+1),-(n-p-l))\in [0,n]\times [0,n],
$$
i.e., $\max\{1,n-p-q\}\leq l\leq\min\{2n-p-q,n+1\}$.
\end{proof}

\begin{remark} \emph{We consider two special cases of Proposition \ref{d-cri}:
\begin{enumerate}[$(i)$]
    \item
For $p=q=n-1$, (\ref{1.4}) is reduced to
$$
    \begin{cases}
      \p A_{0}+(\b{\p}+\p\circ\iota_{\varphi})A_{1}=0,\\
      (\b{\p}+\p\circ\iota_{\varphi})A_{0}+(\b{\p}\circ\iota_{\varphi}+\frac{1}{2}\p\circ\iota^2_{\varphi})A_{1}=0,
    \end{cases}
$$
which is exactly the system \cite[(3.8)]{RwZ} of obstruction equations.
    \item
For $p=q=1$, (\ref{1.4}) is reduced to
$$
    \begin{cases}
      \p\Omega-\p\circ\iota_{\b{\varphi}\varphi}(\Omega)+(\b{\p}+\p\circ\iota_{\varphi})\circ\iota_{\b{\varphi}}(\Omega)=0,\\
      (\b{\p}+\p\circ\iota_{\varphi})(\Omega-\iota_{\b{\varphi}\varphi}(\Omega))+(\b{\p}\circ\iota_{\varphi}+\frac{1}{2}\p\circ\iota^2_{\varphi})\iota_{\b{\varphi}}(\Omega)=0,\\
      \p\circ\iota_{\b{\varphi}}(\Omega)=0.
    \end{cases}
$$
Since $-\iota_{\b{\varphi}\varphi}+\iota_{\varphi}\circ\iota_{\b{\varphi}}=-\iota_{\varphi\b{\varphi}}+\iota_{\b{\varphi}}\circ\iota_{\varphi}$ and $2\iota_{\varphi}\circ\iota_{\b{\varphi}\varphi}\alpha=\iota^2_{\varphi}\circ\iota_{\b{\varphi}}(\alpha)$ for $(1,1)$-form $\alpha$ as shown in \cite[Proposition 2.6]{RwZ},
 one has
$$
    \begin{cases}
     \b{\p}\Omega=\b{\p}\circ(\iota_{\b{\varphi}\varphi}-\iota_{\varphi}\circ\iota_{\b{\varphi}})(\Omega)-\p\circ\iota_{\varphi}(\Omega),\\
     \p\Omega=\p\circ(\iota_{\varphi\b{\varphi}}-\iota_{\b{\varphi}}\circ\iota_{\varphi})(\Omega)-\b{\p}\circ\iota_{\b{\varphi}}(\Omega),\\
     \p\circ\iota_{\b{\varphi}}(\Omega)=0,
    \end{cases}
$$
which is exactly the system of obstruction equations given in \cite[Proposition 2.7]{RwZ}.
\end{enumerate}}
\end{remark}

Unfortunately, the system \eqref{1.4} of obstruction equations consists of too many equations, difficult to solve, and we try to reduce it to one with only
two equations as in Proposition \ref{eff}.

We will use the homogenous notation for a power series here and henceforth. Assuming that $\alpha(t)$ is a power
series of (bundle-valued) $(p,q)$-forms, expanded as
  $$\alpha(t)=\sum_{k=0}^{\infty}\sum_{i+j=k}\alpha_{i,j}t^i\b{t}^j,$$
one uses the notation
\[ \begin{cases}
\alpha(t) = \sum^{\infty}_{k=0} \alpha_k, \\[4pt]
\alpha_k = \sum_{i+j=k} \alpha_{i,j}t^i \overline{t}^j, \\
\end{cases} \]
where $\alpha_k$ is the $k$-order homogeneous part in the expansion
of $\alpha(t)$ and all $\alpha_{i,j}$ are smooth (bundle-valued)
$(p,q)$-forms on $X_0$ with $\alpha(0)=\alpha_{0,0}$.

\begin{lemma}\label{lemma1}
  If $d(e^{\iota_{\varphi}|\iota_{\b{\varphi}}}(\Omega))_{N_1}=0$ for any $N_1\leq N$, then
  $$ (\b{\p}_{\varphi}A_0)_{N_1}=0,\quad (\p A_k+\b{\p}_{\varphi}A_{k+1})_{N_1}=0,\quad k=0,1,2,\ldots$$
  for any $N_1\leq N$.
\end{lemma}
\begin{proof}
From (\ref{2.2.1}), it follows that
  \begin{align*}
 e^{-\iota_{\varphi}}d(e^{\iota_{\varphi}|\iota_{\b{\varphi}}}(\Omega))=\b{\p}_{\varphi}A_0 +\sum_{k=0}^{+\infty}(\p A_k+\b{\p}_{\varphi}A_{k+1}).
\end{align*}
For any $N_1\leq N$,
$$
  0= (e^{-\iota_{\varphi}}d(e^{\iota_{\varphi}|\iota_{\b{\varphi}}}(\Omega)))_{N_1}=(\b{\p}_{\varphi}A_0)_{N_1} +\sum_{k=0}^{+\infty}(\p A_k+\b{\p}_{\varphi}A_{k+1})_{N_1}.
$$
By comparing degrees, we complete the proof.
\end{proof}

As for (\ref{2.2.1}), one can also have
$$d(e^{\iota_{\varphi}|\iota_{\b{\varphi}}}(\alpha))
=e^{\iota_{\varphi}|\iota_{\b{\varphi}}}\circ\left(e^{-\iota_{\varphi}|-\iota_{\b{\varphi}}}\circ e^{\iota_{\varphi}}\circ\left([\p,\iota_{\varphi}]+\b{\p}+\p\right)
    \circ e^{-\iota_{\varphi}}\circ
    e^{\iota_{\varphi}|\iota_{\b{\varphi}}}(\alpha)\right).
$$
A long local calculation shows that
\begin{equation}
 \begin{aligned}\label{2.7}
    &d(e^{\iota_{\varphi}|\iota_{\b{\varphi}}}(\alpha))\\=\ &
    e^{\iota_{\varphi}|\iota_{\b{\varphi}}}\left(\left((\1-\b{\varphi}\varphi)^{-1}-(\1-\b{\varphi}\varphi)^{-1}\b{\varphi}\right)\Finv\left([\p,\iota_{\varphi}]
    +\b{\p}+\p\right)(\1-\b{\varphi}\varphi+\b{\varphi})\Finv\alpha\right).
 \end{aligned}
\end{equation}
Here we use one notation $\Finv$, first introduced in \cite[Subsection 2.1]{RwZ}, to denote the
\emph{simultaneous contraction} on each component of a complex
differential form. For example,
$(\1-\b{\varphi}\varphi+\b{\varphi})\Finv\alpha$ means that the
operator $(\1-\b{\varphi}\varphi+\b{\varphi})$ acts on $\alpha$
simultaneously as:
\begin{align*}
&\ (\1-\b{\varphi}\varphi+\b{\varphi})\Finv(f_{i_1\cdots
i_p\o{j_1}\cdots\o{j_q}}dz^{i_1}\wedge\cdots \wedge dz^{i_p}\wedge
d\b{z}^{j_1}\wedge\cdots \wedge d\b{z}^{j_q})
\\
=&\ f_{i_1\cdots
i_p\o{j_1}\cdots\o{j_q}}(\1-\b{\varphi}\varphi+\b{\varphi})\l
dz^{i_1} \wedge\cdots\wedge(\1-\b{\varphi}\varphi +\b{\varphi})\l
dz^{i_p}
\\&\qquad\quad\quad\wedge(\1-\b{\varphi}\varphi+\b{\varphi})\l
d\b{z}^{j_1}\wedge\cdots\wedge(\1-\b{\varphi}\varphi+\b{\varphi})\l
d\b{z}^{j_q},
\end{align*}
if $\alpha$ is locally expressed by:
$$\alpha=f_{i_1\cdots i_p\o{j_1}\cdots\o{j_q}}dz^{i_1}\wedge \cdots\wedge dz^{i_p}\wedge d\b{z}^{j_1}\wedge \cdots\wedge d\b{z}^{j_q}.$$
This new simultaneous contraction is well-defined since $\varphi(t)$
is a  global $(1,0)$-vector valued $(0, 1)$-form on $X_0$ (See
\cite[Pages $150-151$]{MK}) as reasoned in \cite[Proof of Lemma
2.8]{RZ15}.
Moreover, we know that
$$
  e^{-\iota_{\varphi}|-\iota_{\b{\varphi}}}\circ e^{\iota_{\varphi}}=\left((\1-\b{\varphi}\varphi)^{-1}-(\1-\b{\varphi}\varphi)^{-1}\b{\varphi}\right)\Finv: A^{p,q}(X_0)\to \bigoplus_{i=0}^{\min\{q,n-p\}}A^{p+i,q-i}(X_0).
$$
Thus, by carefully comparing the types of forms in both sides  of
(\ref{2.7}), we have
 \begin{align}\label{2.6}
   \b{\p}_t(e^{\iota_{\varphi}|\iota_{\b{\varphi}}}(\alpha))=e^{\iota_{\varphi}|\iota_{\b{\varphi}}}\left((\1-\b{\varphi}\varphi)^{-1}\Finv([\p,\iota_{\varphi}]+\b{\p})(\1-\b{\varphi}\varphi)\Finv\alpha\right).
 \end{align}
See \cite[Proposition 2.13]{RZ15} and \cite[(2.14)]{RwZ} for more details of \eqref{2.6}.

Here and henceforth one denotes by $(\alpha)^{p,q}$ the $(p,q)$-type part of a $(p+q)$-degree complex differential form $\alpha$.
\begin{proposition}\label{eff}
  The obstruction equation $d(e^{\iota_{\varphi}|\iota_{\b{\varphi}}}(\Omega))=0$ is also equivalent to
  \begin{equation}\label{1.6}
    \begin{cases}
      \sum_{k=0}^{\infty}\left(\b{\p}\circ\frac{\iota^{k-1}_{\varphi}}{(k-1)!}+\p\circ\frac{\iota^{k}_{\varphi}}{k!}\right)\circ\frac{\iota^{k}_{(\1-\b{\varphi}\varphi)^{-1}\b{\varphi}}}{k!}(\tilde{\Omega})=0,\\
      \sum_{k=0}^{\infty}\left(\b{\p}\circ\frac{\iota^{k-1}_{\varphi}}{(k-1)!}+\p\circ\frac{\iota^{k}_{\varphi}}{k!}\right)\circ\frac{\iota^{k-1}_{(\1-\b{\varphi}\varphi)^{-1}\b{\varphi}}}{(k-1)!}(\tilde{\Omega})=0.
    \end{cases}
  \end{equation}
\end{proposition}
\begin{proof}
From the proof of Proposition \ref{d-cri}, it is easy to see that the left-hand side of the first equation in (\ref{1.6}) is $\left(d(e^{\iota_{\varphi}|\iota_{\b{\varphi}}}(\Omega))\right)^{p+1,q}$, while the other one is $\left(d(e^{\iota_{\varphi}|\iota_{\b{\varphi}}}(\Omega))\right)^{p,q+1}$. Thus, (\ref{1.6}) holds if $d(e^{\iota_{\varphi}|\iota_{\b{\varphi}}}(\Omega))=0$.

Conversely, we assume that (\ref{1.6}) holds. By \eqref{2.6} and (\ref{2.2.1}), one compares types of forms to get
\begin{equation}\label{dbt}
\begin{aligned}
 \ &\b{\p}_t e^{\iota_{\varphi}|\iota_{\b{\varphi}}}(\Omega)\\
=\ &e^{\iota_{\varphi}|\iota_{\b{\varphi}}}\circ(\1-\b{\varphi}\varphi)^{-1}\Finv\b{\p}_{\varphi}A_0\\
=\ &e^{\iota_{\varphi}|\iota_{\b{\varphi}}}\circ(\1-\b{\varphi}\varphi)^{-1}\Finv\left(\left(d(e^{\iota_{\varphi}|\iota_{\b{\varphi}}}(\Omega))\right)^{p,q+1}-\sum_{k=0}^{+\infty}\frac{\iota^{k+1}_{\varphi}}{(k+1)!}
    (\p A_k+\b{\p}_{\varphi}A_{k+1})\right).
\end{aligned}
\end{equation}
Similarly, we get
$$
    \p_t e^{\iota_{\varphi}|\iota_{\b{\varphi}}}(\Omega)=e^{\iota_{\varphi}|\iota_{\b{\varphi}}}\circ(\1-\varphi\b{\varphi})^{-1}\Finv\left(\left(d(e^{\iota_{\varphi}|\iota_{\b{\varphi}}}(\Omega))\right)^{p+1,q}-\sum_{k=0}^{+\infty}\frac{\iota^{k+1}_{\b{\varphi}}}{(k+1)!}(\o{\p A_k+\b{\p}_{\varphi}A_{k+1}})\right),
$$
where $\o{\bullet}$ in the last term means that $\p, \b{\p}, \varphi,\b{\varphi}$ are replaced by $\b{\p},\p,\b{\varphi},\varphi$, respectively, while $\Omega$ is taken no conjugation.

  If (\ref{1.6}) holds, i.e., $\left(d(e^{\iota_{\varphi}|\iota_{\b{\varphi}}}(\Omega))\right)^{p+1,q}=0$ and $\left(d(e^{\iota_{\varphi}|\iota_{\b{\varphi}}}(\Omega))\right)^{p,q+1}=0$,
we will prove $$d(e^{\iota_{\varphi}|\iota_{\b{\varphi}}}(\Omega))=0$$ by induction on orders. Obviously, $$d(e^{\iota_{\varphi}|\iota_{\b{\varphi}}}(\Omega))_0=d\Omega_0=0.$$
Now we assume that for any $N_1\leq N$, $d(e^{\iota_{\varphi}|\iota_{\b{\varphi}}}(\Omega))_{N_1}=0$. By Lemma \ref{lemma1}, one has
$$ (\b{\p}_{\varphi}A_0)_{N_1}=0,\quad (\p A_k+\b{\p}_{\varphi}A_{k+1})_{N_1}=0,\quad k=0,1,2,\ldots$$
  for any $N_1\leq N$. For the $(N+1)$-th order, \eqref{dbt} and the induction assumption for any $N_1\leq N$ imply
\begin{align*}
 (\b{\p}_t e^{\iota_{\varphi}|\iota_{\b{\varphi}}}(\Omega))_{N+1}&=\left((e^{\iota_{\varphi}|\iota_{\b{\varphi}}}\circ(\1-\b{\varphi}\varphi)^{-1}\Finv)^{-1}\circ \b{\p}_t e^{\iota_{\varphi}|\iota_{\b{\varphi}}}(\Omega)\right)_{N+1}\\
    &=\left(d(e^{\iota_{\varphi}|\iota_{\b{\varphi}}}(\Omega))\right)^{p,q+1}_{N+1}-\left(\sum_{k=0}^{+\infty}\frac{\iota^{k+1}_{\varphi}}{(k+1)!}
    (\p A_k+\b{\p}_{\varphi}A_{k+1})\right)_{N+1}\\
    &=0.
\end{align*}
Similarly, we have $ (\p_t e^{\iota_{\varphi}|\iota_{\b{\varphi}}}(\Omega))_{N+1}=0$.
 So
  $$d(e^{\iota_{\varphi}|\iota_{\b{\varphi}}}(\Omega))_{N+1}= (\b{\p}_t e^{\iota_{\varphi}|\iota_{\b{\varphi}}}(\Omega))_{N+1}+(\p_t e^{\iota_{\varphi}|\iota_{\b{\varphi}}}(\Omega))_{N+1}=0.$$
  Thus, we complete the proof.
\end{proof}

\begin{remark}
  \emph{For $p=q=1$, if solving $d e^{\iota_{\varphi}|\iota_{\b{\varphi}}}(\Omega)=0$ for the orders $\leq N$, we proceed to the $(N+1)$-th order. One needs to prove
  $$\b{\p}(\varphi\lrcorner\Omega)_{N+1}=(\b{\p} e^{\iota_{\varphi}|\iota_{\b{\varphi}}}(\Omega))^{0,3}_{N+1}=(d e^{\iota_{\varphi}|\iota_{\b{\varphi}}}(\Omega))^{0,3}_{N+1}=0.$$ From (\ref{2.2.1}), we have
$$
    (d e^{\iota_{\varphi}|\iota_{\b{\varphi}}}(\Omega))^{0,3}_{N+1}=\left(\sum_{k=-1}^{+\infty}\frac{\iota^{k+2}_{\varphi}}{(k+2)!}(\p A_k+\b{\p}_{\varphi}A_{k+1})\right)_{N+1}=\left(\frac{\iota^3_\varphi}{3!}\circ\p(\b{\varphi}\lrcorner\Omega)\right)_{N+1}=0,
$$
  where the last equality follows from Lemma \ref{lemma1}.}
\end{remark}

Now we begin to solve (\ref{1.6}) with
two more lemmas. For the resolution of $\p\db$-equations, we need a lemma due to \cite[Theorem $4.1$]{P1} (or \cite[Lemma 3.14]{RwZ}):
\begin{lemma}\label{ddbar-eq}
Let $(X,\omega)$ be a compact Hermitian complex manifold with the
pure-type complex differential forms $x$ and $y$. Assume that the
$\p \db$-equation
\begin{equation}\label{ddb-eq}
\p \db x =y
\end{equation}
admits a solution. Then an explicit solution of the $\p
\db$-equation \eqref{ddb-eq} can be chosen as
$$(\p\db)^*\G_{BC}y,$$
which uniquely minimizes the $L^2$-norms of all the solutions with
respect to $\omega$. Besides, the equalities hold
\[\G_{BC}(\p\db) = (\p\db) \G_{A}\quad \text{and} \quad (\p\db)^*\G_{BC} = \G_{A}(\p\db)^*,\]
where $\G_{BC}$ and $\G_{A}$ are the associated Green's operators of
$\square_{BC}$ and $\square_{A}$, respectively. Here $\square_{BC}$
is defined in \eqref{bc-Lap} and $\square_{{A}}$ is the \emph{second
Kodaira-Spencer operator} (often also called \emph{Aeppli
Laplacian})
$$\square_{{A}}=\p^*\db^*\db\p+\db\p\p^*\db^*+\db\p^*\p\db^*+\p\db^*\db\p^*+\db\db^*+\p\p^*.$$
\end{lemma}

\begin{proof}[Sketch of Proof for Lemma \ref{ddbar-eq}]
We shall use the Hodge decomposition of $\square_{BC}$ on $X$:
$$\label{bc-hod}
A^{p,q}(X)=\ker\square_{BC}\oplus \textrm{Im}~(\p\db)
\oplus(\textrm{Im}~\p^*+\textrm{Im}~\db^*),
$$
whose three parts are orthogonal to each other with respect to the
$L^2$-scalar product defined by $\omega$, combined with the equality
$$\1=\mathbb{H}_{BC}+\square_{BC}\G_{BC},$$
where $\mathbb{H}_{BC}$ is the harmonic projection operator.
Then two observations follow:
\begin{enumerate}[(1)]
    \item \label{case-1}
$\square_{BC}\p\db(\p\db)^*=\p\db(\p\db)^*\square_{BC};$
    \item \label{case-2}
$\G_{BC}\p\db(\p\db)^*=\p\db(\p\db)^*\G_{BC}.$
 \end{enumerate}
It is clear that \eqref{case-1} implies \eqref{case-2},
while the statement \eqref{case-1} is proved by a direct
calculation:
$$\square_{BC}\p\db(\p\db)^*=(\p\db)(\p\db)^*(\p\db)(\p\db)^*=\p\db(\p\db)^*\square_{BC}.$$
Hence, one has
$$(\p\db)(\p\db)^*\G_{BC}y=\G_{BC}(\p\db)(\p\db)^*y=\G_{BC}\square_{BC}y=(\1-\mathbb{H}_{BC})y=y,$$
where $y \in \textrm{Im}~\p \db$ due to the solution-existence of
the $\p\db$-equation.

To see that the solution $(\p\db)^*\G_{BC}y$ is the unique
$L^2$-norm minimum, we resort to the Hodge decomposition of the
operator $\square_{A}$:
\begin{equation}\label{A-H}
A^{p,q}(X)=\ker\square_{A} \oplus (Im~\p +Im~\db ) \oplus
Im~(\p\db)^*,
\end{equation}
where $\ker \square_{A} = \ker (\p \db) \cap \ker \p^* \cap \ker
\db^*$. Let $z$ be an arbitrary solution of the $\p\db$-equation
\eqref{ddb-eq}, which decomposes into three components $z_1 + z_2
+z_3$ with respect to the Hodge decomposition \eqref{A-H} of
$\square_{A}$.
And one is able to
obtain that
\[ z_3 =  \mathbb{G}_{A} (\p\db)^* y = (\p\db)^*\mathbb{G}_{BC} y. \]
Therefore, \[ \|z\|^2 = \|z_1\|^2 +\|z_2\|^2+\|z_3\|^2\geq \|z_3\|^2
= \|(\p\db)^*\mathbb{G}_{BC}y \|^2, \] and the equality holds if and
only if $z_1 =z_2 =0$, i.e., $z=z_3=(\p\db)^*\mathbb{G}_{BC} y$.
\end{proof}

\blemma\label{slvdb-1} Let $X$ be a complex manifold satisfying the $(p,q+1)$- and $(q,p+1)$-th mild $\p\db$-lemmata. Consider the system of equations:
\beq\label{conjugatequ-1}
\begin{cases}
\p x =\db \zeta, \\
\db x =\p \overline{\xi},
\end{cases}\eeq
where $\zeta,\xi$ are $(p+1,q-1)$- and $(q+1,p-1)$-forms on $X$, respectively. The system of
equations \eqref{conjugatequ-1} has a solution if and only if
$$
\begin{cases}
\p\db \zeta=0,\\
\db\p \overline{\xi}=0.
\end{cases}
$$
\elemma

\bproof This lemma is inspired by \cite[Observation 2.11]{RwZ}. The lemmata assumption will produce $\mu\in A^{p,q-1}$ and $\nu\in A^{p-1,q}$, satisfying the system of
equations $$\label{mslmm} \begin{cases}
\p \db \mu = \db \zeta, \\
\db \p \nu = \p \b\xi. \\
\end{cases}  $$
The combined expression
$$\label{cbslt} \db \mu + \p \nu $$
is
our choice for the solution of the system \eqref{conjugatequ-1}.
\eproof

By Lemmata \ref{ddbar-eq} and \ref{slvdb-1}, one has:
\begin{proposition}\label{formal-s}
With the same notations as in Lemmata \ref{ddbar-eq} and \ref{slvdb-1}, the system of
equations \eqref{conjugatequ-1} has a canonical solution
$$
  x=\db(\p\db)^*\G_{BC}\db\zeta-\p(\p\db)^*\G_{BC}\p\b\xi.
$$
\end{proposition}

Now we assume that $X_0$ is a complex manifold satisfies $(q,p+1)$-th and $(p,q+1)$-th mild $\p\b{\p}$-lemmata. The obstruction (\ref{1.6}) can be rewritten as
  \begin{equation}\label{1.7}
    \begin{cases}
      \p\tilde{\Omega}+\sum_{k=1}^{\infty}\left(\b{\p}\circ\frac{\iota^{k-1}_{\varphi}}{(k-1)!}+\p\circ\frac{\iota^{k}_{\varphi}}{k!}\right)\circ\frac{\iota^{k}_{(\1-\b{\varphi}\varphi)^{-1}\b{\varphi}}}{k!}(\tilde{\Omega})=0,\\
      \b{\p}\tilde{\Omega}+\p\circ\iota_{\varphi}(\tilde{\Omega})+\sum_{k=2}^{\infty}\left(\b{\p}\circ\frac{\iota^{k-1}_{\varphi}}{(k-1)!}+\p\circ\frac{\iota^{k}_{\varphi}}{k!}\right)\circ\frac{\iota^{k-1}_{(\1-\b{\varphi}\varphi)^{-1}\b{\varphi}}}{(k-1)!}(\tilde{\Omega})=0.
    \end{cases}
  \end{equation}
Set
$$\tilde{\Omega}':=\tilde{\Omega}+\sum_{k=1}^{\infty}\frac{\iota^{k}_{\varphi}}{k!}\circ\frac{\iota^{k}_{(\1-\b{\varphi}\varphi)^{-1}\b{\varphi}}}{k!}(\tilde{\Omega}).$$
Then (\ref{1.7}) becomes
$$
  \begin{cases}
    \p\tilde{\Omega}'=-\b{\p}\sum_{k=1}^{\infty}\frac{\iota^{k-1}_{\varphi}}{(k-1)!}\circ\frac{\iota^{k}_{(\1-\b{\varphi}\varphi)^{-1}\b{\varphi}}}{k!}(\tilde{\Omega}),\\
    \b{\p}\tilde{\Omega}'=-\p \sum_{k=0}^{\infty} \frac{\iota^{k+1}_{\varphi}}{(k+1)!}\circ\frac{\iota^{k}_{(\1-\b{\varphi}\varphi)^{-1}\b{\varphi}}}{k!}(\tilde{\Omega}).
  \end{cases}
$$
In order to use $(q,p+1)$- and $(p,q+1)$-th mild $\p\b{\p}$-lemmata, we need to prove
\begin{align}\label{2.1}
\begin{cases}
\p\b{\p}\sum_{k=1}^{\infty}\frac{\iota^{k-1}_{\varphi}}{(k-1)!}\circ\frac{\iota^{k}_{(\1-\b{\varphi}\varphi)^{-1}\b{\varphi}}}{k!}(\tilde{\Omega})
=0,\\
\p\b{\p}\sum_{k=0}^{\infty} \frac{\iota^{k+1}_{\varphi}}{(k+1)!}\circ\frac{\iota^{k}_{(\1-\b{\varphi}\varphi)^{-1}\b{\varphi}}}{k!}(\tilde{\Omega})=0,
\end{cases}
\end{align}
at the $(N+1)$-th order if it has been solved for the orders $\leq N$.

Now we prove (\ref{2.1}). Firstly, note that
$$
  \sum_{k=0}^{\infty} \frac{\iota^{k+1}_{\varphi}}{(k+1)!}\circ\frac{\iota^{k}_{(\1-\b{\varphi}\varphi)^{-1}\b{\varphi}}}{k!}(\tilde{\Omega})=\left(e^{\iota_{\varphi}}\circ e^{\iota_{(\1-\b{\varphi}\varphi)^{-1}\b{\varphi}}}(\tilde{\Omega})\right)^{p-1,q+1}=(e^{\iota_{\varphi}|\iota_{\b{\varphi}}}(\Omega))^{p-1,q+1}.
$$
Thus,
$$
\begin{aligned}
\left(\p\b{\p}\sum_{k=0}^{\infty} \frac{\iota^{k+1}_{\varphi}}{(k+1)!}\circ \frac{\iota^{k}_{(\1-\b{\varphi}\varphi)^{-1}\b{\varphi}}}{k!}(\tilde{\Omega})\right)_{N+1}
&=\p\b{\p}\left((e^{\iota_{\varphi}|\iota_{\b{\varphi}}}(\Omega))^{p-1,q+1}\right)_{N+1}\\
 &=\left(\p\b{\p}e^{\iota_{\varphi}|\iota_{\b{\varphi}}}(\Omega)\right)_{N+1}^{p,q+2}\\
 &=\left(\p d e^{\iota_{\varphi}|\iota_{\b{\varphi}}}(\Omega)\right)_{N+1}^{p,q+2}\\
 &=\p\left(d e^{\iota_{\varphi}|\iota_{\b{\varphi}}}(\Omega)\right)_{N+1}^{p-1,q+2}
\end{aligned}
$$
since the obstruction equation is solved for the orders $\leq N$, i.e., $d(e^{\iota_{\varphi}|\iota_{\b{\varphi}}}(\Omega))_{N_1}=0$ for any $N_1\leq N$.
 By Lemma \ref{lemma1}, one has
$$ (\b{\p}_{\varphi}A_0)_{N_1}=0,\quad (\p A_k+\b{\p}_{\varphi}A_{k+1})_{N_1}=0,\quad k=0,1,2,\ldots$$
  for any $N_1\leq N$. It follows from (\ref{2.2.1}) that
  $$\left(d e^{\iota_{\varphi}|\iota_{\b{\varphi}}}(\Omega)\right)_{N+1}^{p-1,q+2}=\left(\sum_{k=-1}^{+\infty}\frac{\iota^{k+2}_{\varphi}}{(k+2)!}(\p A_k+\b{\p}_{\varphi}A_{k+1})\right)_{N+1}=0.$$
  So
  $$\left(\p\b{\p}\sum_{k=0}^{\infty} \frac{\iota^{k+1}_{\varphi}}{(k+1)!}\circ \frac{\iota^{k}_{(\1-\b{\varphi}\varphi)^{-1}\b{\varphi}}}{k!}(\tilde{\Omega})\right)_{N+1}=\p\left(d e^{\iota_{\varphi}|\iota_{\b{\varphi}}}(\Omega)\right)_{N+1}^{p-1,q+2}=0.$$
Hence, we have proved the second equation of (\ref{2.1}). Similarly, by the same argument (i.e., replace all $\varphi$ (resp., $\b{\varphi}$) by $\b{\varphi}$ (resp., $\varphi$)),  the first equation of (\ref{2.1}) also holds.

  By Proposition \ref{formal-s}, one obtains a formal solution of \eqref{1.7} by induction
\begin{equation}\label{express-tO}
\begin{aligned}
\tilde{\Omega}_l&=-\Big(\sum_{k=1}^{\infty}\frac{\iota^{k}_{\varphi}}{k!}\circ\frac{\iota^{k}_{(\1-\b{\varphi}\varphi)^{-1}\b{\varphi}}}{k!}(\tilde{\Omega})\Big)_l
-\db(\p\db)^*\G_{BC}\db\Big(\sum_{k=1}^{\infty}\frac{\iota^{k-1}_{\varphi}}{(k-1)!}\circ\frac{\iota^{k}_{(\1-\b{\varphi}\varphi)^{-1}\b{\varphi}}}{k!}(\tilde{\Omega})\Big)_l\\
 &\quad +\p(\p\db)^*\G_{BC}\p\Big(\sum_{k=0}^{\infty} \frac{\iota^{k+1}_{\varphi}}{(k+1)!}\circ\frac{\iota^{k}_{(\1-\b{\varphi}\varphi)^{-1}\b{\varphi}}}{k!}(\tilde{\Omega})\Big)_l\\
 &=-\Big(\sum_{k=1}^{\min\{q,n-p\}}\frac{\iota^{k}_{\varphi}}{k!}\circ\frac{\iota^{k}_{(\1-\b{\varphi}\varphi)^{-1}\b{\varphi}}}{k!}(\tilde{\Omega})\Big)_l
-\db(\p\db)^*\G_{BC}\db\Big(\sum_{k=1}^{\min\{q,n-p\}}\frac{\iota^{k-1}_{\varphi}}{(k-1)!}\circ\frac{\iota^{k}_{(\1-\b{\varphi}\varphi)^{-1}\b{\varphi}}}{k!}(\tilde{\Omega})\Big)_l\\
 &\quad +\p(\p\db)^*\G_{BC}\p\Big(\sum_{k=0}^{\min\{q,n-p\}} \frac{\iota^{k+1}_{\varphi}}{(k+1)!}\circ\frac{\iota^{k}_{(\1-\b{\varphi}\varphi)^{-1}\b{\varphi}}}{k!}(\tilde{\Omega})\Big)_l.
\end{aligned}
\end{equation}

\begin{remark}\rm
One is possibly able to obtain this (formal) solution \eqref{express-tO} of \eqref{1.7} backwards by the  invertibility of some operator in small $t$ as shown in \cite[Remark 4.6]{lrw}, but it seems that
to figure out this solution explicitly by power series method is indispensable in this process.
\end{remark}

\subsection{Regularity argument}\label{regularity} Here we adopt a strategy for convergence argument \cite{L} suggested by K. Liu, which simplifies
our argument involved in \cite{RwZ,RZ15}.

From the induction expression \eqref{express-tO}, one obtains the formal expression of $\tilde{\Omega}$
\begin{equation}\label{omexp}
\begin{aligned}
\tilde{\Omega}&=-\sum_{i=1}^{\min\{q,n-p\}}\frac{\iota^{i}_{\varphi}}{i!}\circ\frac{\iota^{i}_{(\1-\b{\varphi}\varphi)^{-1}\b{\varphi}}}{i!}(\tilde{\Omega})
-\db(\p\db)^*\G_{BC}\db\sum_{i=1}^{\min\{q,n-p\}}\frac{\iota^{i-1}_{\varphi}}{(i-1)!}\circ\frac{\iota^{i}_{(\1-\b{\varphi}\varphi)^{-1}\b{\varphi}}}{i!}(\tilde{\Omega})\\
 &\quad +\p(\p\db)^*\G_{BC}\p\sum_{i=0}^{\min\{q,n-p\}} \frac{\iota^{i+1}_{\varphi}}{(i+1)!}\circ\frac{\iota^{i}_{(\1-\b{\varphi}\varphi)^{-1}\b{\varphi}}}{i!}(\tilde{\Omega})
 +\Om_0.
\end{aligned}
\end{equation}
Set
$$\begin{aligned}
F=\ &\p(\p\db)^*\G_{BC}\p\sum_{i=0}^{\min\{q,n-p\}} \frac{\iota^{i+1}_{\varphi}}{(i+1)!}\circ\frac{\iota^{i}_{(\1-\b{\varphi}\varphi)^{-1}\b{\varphi}}}{i!}\\
\ &-\sum_{i=1}^{\min\{q,n-p\}}\frac{\iota^{i}_{\varphi}}{i!}\circ\frac{\iota^{i}_{(\1-\b{\varphi}\varphi)^{-1}\b{\varphi}}}{i!}\\
\ &-\db(\p\db)^*\G_{BC}\db\sum_{i=1}^{\min\{q,n-p\}}\frac{\iota^{i-1}_{\varphi}}{(i-1)!}\circ\frac{\iota^{i}_{(\1-\b{\varphi}\varphi)^{-1}\b{\varphi}}}{i!}
\end{aligned}
$$
and write
\begin{equation}\label{global-op}
  \Om_0=(\1-F)\tilde{\Omega}.
\end{equation}
We claim that $\tilde{\Omega}(t)$ converges in H\"older norm as $t\rightarrow 0$ by use of the following two a priori
elliptic estimates: for any complex differential form $\phi$,
$$\label{ee1}
\|\overline{\partial}^*\phi\|_{k-1, \alpha}\leq C_1\|\phi\|_{k,
\alpha}
$$
and
$$\label{ee2}
\|\G_{BC}\phi\|_{k, \alpha}\leq C_{k,\alpha}\|\phi\|_{k-4, \alpha},
$$
where $k>3$ and $C_{k,\alpha}$ depends on only on $k$ and $\alpha$,
not on $\phi$ (cf. \cite[Appendix.Theorem $7.4$]{k} for example).
And one notes that $\varphi(t)$ converges smoothly to zero as $t\rightarrow 0$.
Thus, by \eqref{global-op}, one estimates
$$
\| \Om_0\|_{k, \alpha}\geq
(1-\epsilon_{k,\alpha})
\|\tilde{\Omega} \|_{k, \alpha},$$
where $0<\epsilon_{k,\alpha}\ll 1$ is some constant depending on $k,\alpha$.

Finally, we proceed to the regularity of $\tilde{\Omega}(t)$ since there is possibly no uniform
lower bound for the convergence radius obtained as above in the $C^{k,\alpha}$-norm when $k$ converges to $+\infty$. This argument lies heavily in the elliptic estimates
\cite[Appendix.\S 8]{k}, \cite{dn} and also \cite[Subsection 3.2]{RwZ}.

Without loss of generality, we just consider the equation:
$$\label{reg-eqn-1}
\begin{aligned}
\square\tilde{\Omega}&=-\square\sum_{k=1}^{\min\{q,n-p\}}\frac{\iota^{k}_{\varphi}}{k!}\circ\frac{\iota^{k}_{(\1-\b{\varphi}\varphi)^{-1}\b{\varphi}}}{k!}(\tilde{\Omega})
-\db \db^*\db(\p\db)^*\G_{BC}\db\sum_{k=1}^{\min\{q,n-p\}}\frac{\iota^{k-1}_{\varphi}}{(k-1)!}\circ\frac{\iota^{k}_{(\1-\b{\varphi}\varphi)^{-1}\b{\varphi}}}{k!}(\tilde{\Omega})\\
 &\quad +\square\p(\p\db)^*\G_{BC}\p\sum_{k=0}^{\min\{q,n-p\}} \frac{\iota^{k+1}_{\varphi}}{(k+1)!}\circ\frac{\iota^{k}_{(\1-\b{\varphi}\varphi)^{-1}\b{\varphi}}}{k!}(\tilde{\Omega})
\end{aligned}
$$
by applying the $\db$-Laplacian
$\square=\db^*\db+\db\db^*$ to the expression formula \eqref{omexp} and omitting the lower-order term $\square\Om_0$ in this expression.
By replacing the roles of
$$\iota_{\varphi}\circ
\iota_{(\1-\b{\varphi}\varphi)^{-1}\b{\varphi}},\ \iota_{(\1-\b{\varphi}\varphi)^{-1}\b{\varphi}},\ \iota_{\varphi}+
\frac{1}{2} \iota_{\varphi}\circ \iota_{\varphi}\circ
\iota_{(\1-\b{\varphi}\varphi)^{-1}\b{\varphi}}$$
in
the analogous strongly elliptic second-order pseudo-differential equation in the regularity argument of \cite[Subsection 3.2]{RwZ}
$$
\label{reg-eqn-2}
\begin{aligned}
\square\tilde{\Omega}(t)
 =&-\square\Big( \iota_{\varphi}\circ
\iota_{(\1-\b{\varphi}\varphi)^{-1}\b{\varphi}} \big(
\tilde{\Omega}(t)\big) \Big) -\db \db^*\db(\p\db)^* \G_{BC} \db
\Big( \iota_{(\1-\b{\varphi}\varphi)^{-1}\b{\varphi}} \big(
\tilde{\Omega}(t) \big) \Big)\\
& + \square\p (\p\db)^* \G_{BC} \p \Big( \big( \iota_{\varphi}+
\frac{1}{2} \iota_{\varphi}\circ \iota_{\varphi}\circ
\iota_{(\1-\b{\varphi}\varphi)^{-1}\b{\varphi}}\big)\tilde{\Omega}(t)
\Big),
\end{aligned}
$$
by
$$\sum_{i=1}^{\min\{q,n-p\}}\frac{\iota^{i}_{\varphi}}{i!}\circ\frac{\iota^{i}_{(\1-\b{\varphi}\varphi)^{-1}\b{\varphi}}}{i!},\
\sum_{i=1}^{\min\{q,n-p\}}\frac{\iota^{i-1}_{\varphi}}{(i-1)!}\circ\frac{\iota^{i}_{(\1-\b{\varphi}\varphi)^{-1}\b{\varphi}}}{i!},\
\sum_{i=0}^{\min\{q,n-p\}} \frac{\iota^{i+1}_{\varphi}}{(i+1)!}\circ\frac{\iota^{i}_{(\1-\b{\varphi}\varphi)^{-1}\b{\varphi}}}{i!},$$
respectively,
one proves the following result. For each $l=1,2,\cdots$, choose a
smooth function $\eta^l(t)$ with values in $[0,1]$:
$$\label{eta-l}
\eta^l(t)\equiv
    \begin{cases}
      1,\ \text{for $|t|\leq (\frac{1}{2}+\frac{1}{2^{l+1}})r$};\\
      0,\ \text{for $|t|\geq (\frac{1}{2}+\frac{1}{2^{l}})r$},
    \end{cases}
$$
where $r$ is a positive constant to be determined.
Inductively, by Douglis-Nirenberg's
interior estimates \cite[Appendix.Theorem 2.3]{k}, \cite{dn}, for any $l=1,2,\cdots$,
$\eta^{2l+1}\tilde{\Omega}(t)$ is $C^{k+l, \alpha}$, where
 $r$ can be chosen independent of $l$. Since $\eta^{2l+1}(t)$ is
identically equal to $1$ on $|t|<\frac{r}{2}$ which is independent
of $l$, $\tilde{\Omega}(t)$ is $C^{\infty}$ on $X_0$ with
$|t|<\frac{r}{2}$.  Then $\tilde{\Omega}(t)$ can be considered as a real analytic family of $(p,q)$-forms in $t$
 and thus is smooth on $t$.

\section{Deformation invariance of Bott-Chern numbers}\label{bcinv}
The main goal of this section is to study deformation invariance of Bott-Chern numbers on complex manifolds:
\begin{theorem}\label{inv-pq}
If the reference fiber $X_0$ satisfies the $(p,q+1)$- and $(q,p+1)$-th mild $\p\bar{\p}$-lemmata and the deformation invariance of the $(p-1,q-1)$-Aeppli number $h^{p-1,q-1}_{A}(X_t)$ holds,
then $h^{p,q}_{BC}(X_t)$ are deformation invariant.
\end{theorem}
As a direct corollary, one obtains:
\begin{corollary}\label{inv-p0-bc}
If the reference fiber $X_0$ satisfies the $(p,1)$-th mild $\p\bar{\p}$-lemma,
then $h^{p,0}_{BC}(X_t)$  and $h^{0,p}_{BC}(X_t)$ are deformation invariant.
\end{corollary}

Resorting to the calculations for the Hodge and Bott-Chern numbers of manifolds in the Kuranishi family of the Iwasawa manifold (cf. \cite[Appendix]{A}),
we find the following example that neither the deformation invariance of the $(p,0)$- nor $(0,p)$-Bott-Chern numbers
is true when the condition that the $(p,1)$-th mild $\p\bar{\p}$-lemma does not hold on the reference fiber
in Corollary \ref{inv-p0-bc}. It indicates that the condition involved may not be omitted
in order for the deformation invariance of $(p,0)$- and $(0,p)$-Bott-Chern numbers.

Let $\mathbb{I}_3$ be the Iwasawa manifold of complex dimension $3$ with $\eta^1,\eta^2,\eta^3$
denoted by the basis of the holomorphic one form $H^0(\mathbb{I}_3,\Omega^1)$ of $\mathbb{I}_3$, satisfying
the relation \[ d\eta^1 =0,\ d \eta^2=0,\ d\eta^3 = - \eta^1 \w \eta^2. \]
And the convention $\eta^{12\bar{1}\bar{3}}:= \eta^1 \w \eta^2 \w \overline{\eta}^1 \w \overline{\eta}^3$
will be used for simplicity.

\begin{example}[The cases $(p,q)=(2,0)$ and $(0,2)$]\label{ex20}
The injectivity of $\iota^{2,1}_{BC,\p}$ does not hold on $\mathbb{I}_3$
and in these cases $h^{2,0}_{BC}(X_t)$ and $h^{0,2}_{BC}(X_t)$ are deformation variant.
\end{example}

\begin{proof}
It is easy to check that the left-invariant $(2,1)$-form
\[ \p \eta^{3\bar{1}} = - \eta^{12\bar{1}}\]
stands for a non-trivial Bott-Chern class but a trivial class in $H^{2,1}_{\p}(\mathbb{I}_3)$,
which indicates non-injectivity of $\iota^{2,1}_{BC,\p}$.
The deformation variance of $h^{2,0}_{BC}(X_t)$ and $h^{0,2}_{BC}(X_t)$ can be got from \cite[Appendix]{A}.
\end{proof}

Now let us describe our basic philosophy to consider the deformation invariance of Bott-Chern numbers briefly. The
Kodaira-Spencer's upper semi-continuity theorem (\cite[Theorem 4]{KS}) tells us that the function
$$t\longmapsto h^{p,q}_{BC}(X_t)=\dim_{\mathbb{C}}H^{p,q}_{BC}(X_t,\mathbb{C})$$
is always upper semi-continuous for $t\in B$ and thus,
to approach the deformational invariance of $h^{p,q}_{BC}(X_t)$, we only need to obtain the lower
semi-continuity. Here our main strategy is a modified iteration
procedure, originally from \cite{LSY} and developed in
\cite{Sun,SY,RZ,lry}, which is to look for an injective extension map from
$H^{p,q}_{BC}(X_0)$ to $H^{p,q}_{BC}(X_t)$.
More precisely, for the unique harmonic representative $\sigma_0$ of the initial Bott-Chern
cohomology class in $H^{p,q}_{BC}(X_0)$, we try to construct a convergent power series
$$\label{ps} \sigma_t=\sigma_0+\sum_{j+k=1}^\infty t^k t^{\bar{j}}\sigma_{k\bar{j}}\in
A^{p,q}(X_0),$$
with $\sigma_t$ varying smoothly on $t$ such that for each small $t$:
\begin{enumerate}[$(i)$]
\item\label{S1} $e^{\iota_{\varphi}|\iota_{\overline{\varphi}}}(\sigma_t)\in A^{p,q}(X_t)$
is $d$-closed with respect to the differential structure on $X_t$ with the induced family $\varphi$ of Beltrami differentials;
\item\label{S2} The extension map $H^{p,q}_{BC}(X_0) \rightarrow H^{p,q}_{BC}(X_t):[\sigma_0]_{d} \mapsto
[e^{\iota_{\varphi}|\iota_{\overline{\varphi}}}(\sigma_t)]_{d}$ is injective.
\end{enumerate}

Obviously, \eqref{S1} amounts to Theorem \ref{pq-sm-ext}; To guarantee \eqref{S2}, it suffices to prove:
\begin{proposition}\label{pq-pq-1}
If the $d$-extension of $H^{p,q}_{BC}(X_0)$
as in  Theorem \ref{pq-sm-ext} holds for a complex manifold $X_0$,
then the deformation invariance of $h^{p-1,q-1}_{A}(X_t)$
assures that the extension map
$$H^{p,q}_{BC}(X_0) \rightarrow H^{p,q}_{BC}(X_t):[\sigma_0]_{d} \mapsto [e^{\iota_{\varphi}|\iota_{\overline{\varphi}}}(\sigma_t)]_{d}$$
is injective.
\end{proposition}
\begin{proof}
Here we follow an idea in \cite[Proposition 3.15]{RZ15}.
Let's fix a family of smoothly varying Hermitian metrics
$\{\omega_t\}_{t \in  B }$ for the infinitesimal
deformation $\pi: \mathcal{X} \rightarrow  B $ of
$X_0$. Thus, if the Aeppli numbers $h^{p-1,q-1}_{A}(X_t)$ are deformation
invariant, the Green's operator $\G_{A,t}$, acting on the
$A^{p-1,q-1}(X_t)$, depends differentiably with respect to $t$ from
\cite[Theorem 7]{KS} by Kodaira and Spencer. Using this, one ensures
that this extension map can not send a nonzero class in
$H^{p,q}_{BC}(X_0)$ to a zero class in $H^{p,q}_{BC}(X_t)$.

If we suppose that
$$e^{\iota_{\varphi(t)}|\iota_{\overline{\varphi(t)}}}(\sigma_t)=\p_t\db_t \eta_t$$
for some $\eta_t\in A^{p-1,q-1}(X_t)$ when $t \in  B
\setminus \{0\}$, the Hodge decomposition of Bott-Chern Laplacian and the
commutativity
$$\G_{BC}(\p\db) = (\p\db) \G_{A}$$
in Lemma \ref{ddbar-eq} yield that
\begin{align*}
e^{\iota_{\varphi(t)}|\iota_{\overline{\varphi(t)}}}(\sigma_t)
 &=\p_t\db_t \eta_t=\big(\mathbb{H}_{BC,t}+ \square_{BC,t} \mathbb{G}_{BC,t}\big)\p_t \db_t(\eta_t)\\
 &=\mathbb{G}_{BC,t}\square_{BC,t}\p_t \db_t(\eta_t)\\
 &=\mathbb{G}_{BC,t}\p_t \db_t\db_t^*\p_t^*\p_t \db_t(\eta_t)\\
 &=\p_t \db_t\mathbb{G}_{A,t}\db_t^*\p_t^*(e^{\iota_{\varphi(t)}|\iota_{\overline{\varphi(t)}}}(\sigma_t)),
\end{align*}
where $\mathbb{H}_{BC,t}$, $\square_{BC,t}$ are the harmonic projectors and
the Bott-Chern Laplacian with respect to $(X_t,\omega_t)$, respectively.
Let $t$ converge to $0$ on both sides of the equality
\[ e^{\iota_{\varphi(t)}|\iota_{\overline{\varphi(t)}}}(\sigma_t)
=\p_t \db_t\mathbb{G}_{A,t}\db_t^*\p_t^*(e^{\iota_{\varphi(t)}|\iota_{\overline{\varphi(t)}}}(\sigma_t)),\]
which turns out that $\sigma_0$ is $\p\db$-exact on the reference fiber
$X_0$. Here we use the fact that the Green's operator $\mathbb{G}_{A,t}$ depends
differentiably with respect to $t$.
\end{proof}

\begin{acknowledgements}
The authors would like to express their gratitude to Professor Kefeng Liu for his constant help,
especially in the formal convergence argument here, and to Professor L. Ugarte for pointing out an example to us.
This work started from the first author's visit to Institut Fourier, Universit\'{e} Grenoble Alpes from March to June 2017,
and was mostly completed during his visit to Institute of Mathematics, Academia Sinica since September 2017.
He would like to thank both institutes for their hospitality and excellent working space.
The third author is grateful to Professors Fangyang Zheng and Bo Guan for their help and interests
during his visit to the Ohio State University in November 2017.
\end{acknowledgements}

\end{document}